\documentclass[11pt,amsfonts]{article}

\usepackage{latexsym}
\usepackage{indentfirst}
\usepackage{graphicx}
\usepackage{amsmath,amsfonts}
\usepackage{amssymb}
\usepackage{amsthm}
\usepackage{tikz}
\usepackage{bm,bbm}
\usepackage{accents}
\usepackage{changepage}
\usepackage{color}
\usepackage{bigints}
\usepackage{enumerate}
\usepackage{nccmath}
\usepackage{yhmath}
\usepackage{dirtytalk}
\usepackage{comment}
\usepackage{enumitem}
\usepackage{cleveref}

\newtheorem{prop}{Proposition}[section]

\newtheorem{coro}[prop]{Corollary}

\newtheorem{lemma}[prop]{Lemma}
\newtheorem{definition}[prop]{Definition}
\newtheorem{corollary}[prop]{Corollary}
\newtheorem{theorem}[prop]{Theorem}
\newtheorem{remark}[prop]{Remark}

\def\real{{\mathord{{\rm I\kern-2.8pt R}}}}        
\def\inte{{\mathord{{\rm I\kern-2.8pt N}}}}

\def\sZZ{{\rm Z\kern-2.8ptem{}Z}}

\def\z{{\mathchoice
		{\sZZ}
		{\sZZ}
		{\rm Z\kern-0.30em{}Z}
		{\rm Z\kern-0.25em{}Z} }}
\def\sQQ{{\kern 0.27em \vrule height1.45ex width0.03em depth0em
		\kern-0.30em \rm Q}}
\def\qu{{\mathchoice
		{\sQQ}
		{\sQQ}
		{\kern 0.225em \vrule height1.05ex width0.025em depth0em \kern-0.25em \rm Q}
		{\kern 0.180em \vrule height0.78ex width0.020em depth0em \kern-0.20em \rm Q}
}}
\def\sCC{{\kern 0.27em \vrule height1.45ex width0.03em depth0em
		\kern-0.30em \rm C}}
\def\complex{{\mathchoice
		{\sCC}
		{\sCC}
		{\kern 0.225em \vrule height1.05ex width0.025em depth0em \kern-0.25em \rm C}
		{\kern 0.180em \vrule height0.78ex width0.020em depth0em \kern-0.20em \rm C}
}}

\newcommand{\ba}{\begin{array}}
	\newcommand{\ea}{\end{array}}
\newcommand{\be}{\begin{equation}}
	\newcommand{\ee}{\end{equation}}
\newcommand{\bea}{\begin{eqnarray}}
	\newcommand{\eea}{\end{eqnarray}}
\newcommand{\beaa}{\begin{eqnarray*}}
	\newcommand{\eeaa}{\end{eqnarray*}}

\newcommand{\eps}{\varepsilon}

\def\z{\zeta}

\font\tenmath=msbm10 \font\sevenmath=msbm7 \font\fivemath=msbm5
\newfam\mathfam \textfont\mathfam=\tenmath
\scriptfont\mathfam=\sevenmath \scriptscriptfont\mathfam=\fivemath

\def \={{\buildrel {\rm (law)} \over =}}

\newcommand{\basa}{\begin{assumption}}
	\newcommand{\easa}{\end{assumption}}

\newcommand{\bas}{\begin{assum}}
	\newcommand{\eas}{\end{assum}}

\newcommand{\ignore}[1]{}
\begin{document}
	
	\renewcommand{\thefootnote}{\fnsymbol{footnote}}
	\renewcommand{\thefootnote}{\fnsymbol{footnote}}
	
	\title{Quartic variation of the solution to the semilinear stochastic heat equation: limit behavior and asymptotic independence with respect to the   data  }
	
	\author{I. C\^impean \footnote{
			I. C\^impean acknowledges support from the Ministry of Research, Innovation and Digitalization (Romania), grant CF-194-PNRR-III-C9-2023.  } $^{1,2}$ \hskip0.2cm 
		Yassine Nachit $^{3}$ \hskip0.2cm  
		Ciprian A. Tudor \footnote{	 C. Tudor also acknowledges support from   the  ANR project SDAIM 22-CE40-0015, the MATHAMSUD grant 24-MATH-04 SDE-EXPLORE, Japan Science and Technology Agency CREST (grant JPMJCR2115) and  by the Ministry of Research, Innovation and Digitalization (Romania), grant CF-194-PNRR-III-C9-2023.} $^{3}$ \vspace*{0.1in} \\
		$^{1}$ University of Bucharest\\
		Faculty of Mathematics and Computer Science\\
		$^{2}$   Simion Stoilow Institute of
		Mathematics of the Romanian Academy\\
		RO-014700 Bucharest Romania.\\
		\quad iulian.cimpean@unibuc.ro\\
		\vspace*{0.1in} \\
		$^{3}$  Universit\'e de Lille, CNRS\\
		Laboratoire Paul Painlev\'e UMR 8524\\
		F-59655 Villeneuve d'Ascq, France.\\
		\quad yassine.nachit@univ-lille.fr\\
		\quad ciprian.tudor@univ-lille.fr\vspace*{0.1in}}
	\maketitle
	
	\begin{abstract}
		This work concerns the limit behavior of the quartic variation (i.e., the power variation of order four) with respect to the time variable of the solution to the semilinear stochastic heat equation with space-time white noise. 
		In a first step, we prove that this sequence satisfies a Central Limit Theorem and we deduce a similar result for the viscosity parameter estimator associated with the quartic variation. 
		Then, by using a recent variant of the Stein-Malliavin calculus, we analyze the asymptotic independence between the quartic variation (as well as the associated viscosity parameter estimator) and the data used to construct it.   
	\end{abstract}
	
	{\bf 2010 AMS Classification Numbers:}  60G15, 60H05, 60G18.

	\vskip0.3cm

	
	\section{Introduction}
	The Stein-Malliavin calculus, initially introduced in \cite{NP1}, constitutes a probabilistic theory which allows to evaluate the distance between the probability distributions of two random variables in terms of certain operators from Malliavin calculus. We refer to the monograph \cite{NP-book} for a detailed exposition of this theory. A more recent variant of this theory has been developed in the references \cite{Pi} and \cite{T2}.  Its purpose is to obtain bounds, again in terms of the Malliavin operators, for the distance between the distribution $\mathbb{P}_{(X,\mathbf{Y})}$ of a random vector $(X, \mathbf{Y})$ ($X$ is a scalar random variable and $\mathbf{Y}=(Y_{1},...,Y_{d})$ is a $d$-dimensional random vector) and the law of the random vector $(Z, \mathbf{Y})$ with $Z$ Gaussian and independent of $\mathbf{Y}$, i.e., $\mathbb{P}_{Z}\otimes \mathbb{P}_{\mathbf{Y}}$. 
	More precisely, if we assume that $X$ is centered and together with $Y_{j}, j=1,...,d$, belong to the Sobolev space $\mathbb{D}^{1,4}$ (in particular, they are differentiable in the Malliavin sense), then the main result in \cite{T2} states that, if $d_{W}$ stands for the Wasserstein distance and $Z \sim N(0, \sigma ^{2})$,
	\begin{eqnarray}\label{intro-1}
		&&	d_{W} \left( \mathbb{P}_{(X, \mathbf{Y})}, \mathbb{P}_{Z} \otimes \mathbb{P}_{\mathbf{Y}}\right) \\
		&&\leq  C \left[ \sqrt{ \mathbb{E} \left[\left( \langle D(-L) ^{-1}X, DX\rangle -\sigma ^{2} \right) ^{2}\right]}+ \sqrt{ \sum_{j=1}^{d} \mathbb{E} \left[\langle D(-L) ^{-1}X, DY_{j}\rangle ^{2}\right]}\right],\nonumber
	\end{eqnarray}
	where $D$ is the Malliavin derivative and $(-L)^{-1}$ is the pseudo-inverse of the Ornstein-Uhlenbeck operator (see the Appendix for their definitions).  
	The inequality (\ref{intro-1}) extends the classical Stein-Malliavin bound obtained in \cite{NP1} which states that if $X \in \mathbb{D} ^{1, 4}$ and $ Z\sim N(0, \sigma ^{2})$, then 
	\begin{equation*}
		d_{W} \left( \mathbb{P}_{X}, \mathbb{P}_{Z}\right) \leq C  \sqrt{ \mathbb{E} \left[\left( \langle D(-L) ^{-1}X, DX\rangle -\sigma ^{2} \right) ^{2}\right]}.
	\end{equation*}
	We notice that in the right-hand side of (\ref{intro-1}), the first term quantifies the distance between $X$ and the Gaussian random variable $Z$, while the second summand is interpreted as a  measure  for the dependence between $X$ and the components of $\mathbf{Y}$. 
	
	This variant of the Stein-Malliavin calculus allows to analyze the asymptotic independence of random sequences. Let $(X_{N}, N\geq 1)$ be a sequence of random variables which converges in law to $Z\sim N(0, \sigma ^{2})$ and let $(\mathbf{Y}_{N}, N\geq 1)=((Y_{N,1},...,Y_{N,d}), N\geq 1)$ be an arbitrary sequence of random vectors with components regular enough in the Malliavin sense. The above estimate (\ref{intro-1}) permits evaluating the following Wasserstein distance
	\begin{equation*}
		d_{W} \left( \mathbb{P}_{(X_{N}, \mathbf{Y}_{N})}, \mathbb{P}_{Z}\otimes \mathbb{P}_{\mathbf{Y}_{N}}\right), \quad \mbox{for each } N\geq 1.
	\end{equation*}
	When this distance converges to zero, we can conclude that the random sequences $(X_{N}, N\geq 1)$  and $(\mathbf{Y}_{N}, N\geq 1)$ are {\it asymptotically independent}, meaning that, for large $N$, the distribution of the vector $(X_{N}, \mathbf{Y}_{N})$ is approximately the product of its two marginals. This method has been employed in  \cite{T2} or   \cite{TZ} to analyze the asymptotic independence in the context of random matrices, central or non-central limit theorems, or stochastic partial differential equations. 
	
	Here, the purpose is to apply this theory to the power variation of the solution to the stochastic heat equation and to certain related statistical estimators. The parameter estimation for stochastic partial differential equations (SPDEs in the sequel), in particular   via power variations,  constitutes nowadays a very active research direction. The reader may consult the website https://sites.google.com/view/stats4spdes for a  vast bibliography on this topic.  The reader may also consult \cite{Cia} or the recent monograph \cite{T3}  for a survey of the estimation techniques for SPDEs with additive Gaussian noise. Other  references on statistical inference for SPDEs via power variation are, among others,  \cite{BT1}, \cite{CD}, \cite{CDK}, \cite{CKP},  \cite{GT}, \cite{HiTr}, \cite{MahTud2} or \cite{PoTr}.  
	
	Actually, we consider the semilinear stochastic heat equation driven by a space-time white noise, i.e., the model \eqref{1} introduced in Section \ref{sec211}. 
	A strongly consistent  estimator for the viscosity parameter $\theta >0$ (it is sometimes called drift parameter in the literature, but we choose to call it viscosity parameter in this work in order to avoid confusion with the drift coefficient $b$) in \eqref{1}, more precisely the sequence $\widehat{\theta}_{N,x}(u)$ given by \eqref{est}, can be constructed by using the quartic variation in time of the solution (the random sequence denoted by $V_{N, x}(u)$  and given by (\ref{vnx})).  
	
	This estimator is actually defined in terms of the observations 
	\begin{equation*}
		( u(t_{i}, x), i=1,...,N), \quad t_{i}= \frac{i}{N}, 1\leq i \leq N, \quad \mbox{whilst } x\in \mathbb{R} \mbox{ is being fixed}. 
	\end{equation*}
	We consider a random vector $\mathbf{Y}_{N,x}(u)$ of the form 
	\begin{equation*}
		\mathbf{Y}_{N,x}(u) =(u(t_{i}, x), i\in J_{N}), \quad \mbox{where } J_{N} \mbox{ is a subset of } \{1,2,...,N\}.
	\end{equation*}
	So, the vector $\mathbf{Y}_{N,x}(u)$ contains a part of the available observations. 
	Our basic assumption is $	1\leq card(J_{N}):=m(N)\leq N$, where $card (A)$ denotes the cardinality of the set $A$. 
	In this way, the information included in $ \mathbf{Y}_{N, x}(u)$ may contain all the data (when $J_{N}=\{1,...,N\}$ and so $m(N)=N$)  or only a strict part of the data (when $m(N)<N$). 
	We shall first evaluate the following distance
	\begin{equation*}
		\varphi(N):= d_{W}\left( \mathbb{P}_{(U_{N,x}(u), \mathbf{Y}_{N,x}(u))},  \mathbb{P}_{Z}\otimes \mathbb{P}_{\mathbf{Y}_{N,x}(u)}\right), \quad N\geq 1,
	\end{equation*}
	where $(U_{N,x}(u), N\geq 1)$ is the renormalized quartic variation in time of the solution to (\ref{1}), i.e., the sequence given by (\ref{19a-3}). In some situations, depending on the number of observations contained in the random vector $\mathbf{Y}_{N,x}(u)$, the quantity $\varphi(N)$ converges to zero as $N$ tends to infinity. When this happens, we deduce that the  law of the vector $(U_{N,x}(u), \mathbf{Y}_{N,x})$ is close, for $N$ large, to the law of the vector with the same marginals but with independent components. In other words, $U_{N, x}(u)$ is asymptotically independent of the observation vector $\mathbf{Y}_{N,x}(u)$. 
	In order to deduce that also the (renormalized) associated estimator $\widehat{\theta}_{N,x}(u)$ is asymptotically independent of the data $ \mathbf{Y}_{N,x}(u)$, we need to introduce and analyze some weaker notions of asymptotic independence; we do this systematically in \Cref{sec4}.
	Once we achieve this, we can conclude that the extreme values of the measured data do not affect the asymptotic distribution of the estimator. 
	A more detailed discussion around the asymptotic independence can be found in the reference \cite{DN}.
	
	We organized our work as follows. Section 2 contains some preliminaries on the solution to the stochastic heat equation. We also recall some known facts concerning its temporal quartic variation and the associated estimator for the viscosity parameter. In Section 3 we analyze the limit behavior of this quartic variation and, in particular, we derive its rate of convergence to the Gaussian law under the Wasserstein distance. In Sections 4 and 5  we discuss the asymptotic independence of the temporal quartic variation and of the associated viscosity  parameter estimator with respect to the data from which these sequences are constructed. The last section is the Appendix, where we included the basic tools from Malliavin calculus and Wiener chaos needed in our work.

	\section{Preliminaries}
	In this section, we present the basic properties of the solution to the semilinear stochastic heat equation on $\mathbb{R}$.
	We also introduce the quartic variation of the solution with respect to its temporal variable and the connection to the viscosity parameter estimator. 
	
	\subsection{The solution to the semilinear stochastic heat equation}\label{sec211}
	Our work concerns the following SPDE 
	\begin{equation}\label{1}
		\frac{\partial u}{\partial t} (t,x)= \frac{\theta}{2} \Delta u(t,x) + b(u(t,x))+ \dot{W}(t,x), \mbox{ for } t> 0,\; x\in \mathbb{R},
	\end{equation}
	with vanishing initial value $u(0, x)=0$ for every $x\in \mathbb{R}$. 	
	In \eqref{1}, $W$ is a space-time Gaussian white noise which is defined as a  centered Gaussian field $\left( W(t,A), t\geq 0, A\in \mathcal{B}_{b} (\mathbb{R})\right)$ with covariance
	\begin{equation}\label{covW}
		\mathbb{E} \left[W(t,A) W(s, B)\right] = (t\wedge s) \lambda (A \cap B), \hskip0.3cm s,t \geq 0,\; A, B \in \mathcal{B}_{b} (\mathbb{R}),
	\end{equation}
	where $\lambda $ denotes the Lebesgue measure on $\mathbb{R}$ and $\mathcal{B}_{b} (\mathbb{R})$ is  the collection of bounded Borel subsets of $\mathbb{R}$. 
	Furthermore, the drift  coefficient $ b:\mathbb{R} \to \mathbb{R}$ is assumed to be globally Lipschitz continuous,  i.e., there exists $L>0$ such that
	\begin{equation}
		\label{cc} \vert b(x)-b(y)\vert \leq L \vert x-y\vert, \quad x, y \in \mathbb{R}.
	\end{equation}
	Finally, $\theta>0$ is a constant parameter whose estimation is the central topic of this work.
	
	As usual, the solution to (\ref{1}) is understood in the mild sense, i.e., $u$ is a solution to (\ref{1}) if it satisfies, for every $t\geq0, x\in \mathbb{R}$,
	\begin{equation}\label{mild}
		u(t,x)= \int_{0} ^{t} \int_{\mathbb{R}} G(\theta (t-s), x-y)W(ds, dy) +  \int_{0} ^{t} \int_{\mathbb{R}} G(\theta (t-s), x-y)b(u(s,y)) dy ds.
	\end{equation}
	In (\ref{mild}), the integral $ W(ds, dy)$ is a Wiener integral with respect to the Gaussian random field $W$ (see e.g., \cite{T3}) and $G$ is the Green kernel associated to the heat equation, i.e., 
	\begin{equation}\label{G}
		G(t,x)= \frac{1}{\sqrt{2\pi t}} e ^{-\frac{ x^{2}}{2t}} \mbox{ for } t>0, \; x\in \mathbb{R}.
	\end{equation}
	It  is known (see e.g., \cite{DS1}) that under the assumption (\ref{cc}), the SPDE (\ref{1}) admits a unique mild solution $ u\in L^{2}(\mathbb{R}_{+}\times \mathbb{R})$. Moreover, this solution satisfies, for every $T>0$ and for every $p\geq 1$, 
	\begin{equation*}
		\sup_{t\in [0,T], x\in \mathbb{R}	}\mathbb{E} \left[\vert u(t,x) \vert ^{p}\right] \leq C_{T,p},
	\end{equation*}
	where $ C_{T, p}$ is a strictly positive constant depending on $T$ and $p$. We refer to Theorem 4.2.1 in  \cite{DS1} for these results.
	
	\subsection{The quartic variation and the estimator for the viscosity parameter}\label{sec21}
	Let us introduce the quartic variation in time of the mild solution (\ref{mild}) and the associated estimator for the viscosity parameter $\theta$ in (\ref{1}).  Let $N\geq 1$ and let 
	\begin{equation}\label{ti}
		t_{i}= \frac{i}{N}, \mbox{ for } i=0,1,...,N,
	\end{equation}
	be an equidistant  partition of the unit interval $[0,1]$.  
	For a fixed $x\in \mathbb{R}$, we set
	\begin{equation}\label{vnx}
		V_{N, x}(u)= \sum_{i=0} ^{N-1} \left( u( t_{i+1}, x)-u(t_{i}, x) \right) ^{4}, \quad N\geq 1.
	\end{equation}
	The sequence $(V_{N, x}(u), N \geq 1)$ constitutes {\it  the quartic variation in time } of the solution to the semilinear stochastic heat equation (\ref{1}). 
	It is known that in the case of the linear stochastic heat equation (i.e., $b\equiv 0$ in (\ref{mild})), the solution admits a nontrivial quartic variation in time, i.e., $V_{N, x}(u)$ converges, as $N\to\infty$, to an explicit  constant. We refer to  \cite{PoTr}, \cite{Sw}, or \cite{T3}. 
	First, we extend this result to the SPDE \eqref{1}, namely, we show that 
	\begin{equation}\label{19a-2}
		\lim_{N\to \infty}V_{N, x}(u)=\frac{6}{\pi \theta} \quad \mbox{ in  }L ^{1}(\mathbb{P}).
	\end{equation} 
	We will also prove that $ V_{N, x}(u)$ satisfies a Central Limit Theorem (CLT),  i.e. 
	\begin{equation}\label{19a-3}
		U_{N,x}(u):=	\sqrt{N} \left( V_{N, x}(u)-  \frac{6}{\pi \theta}\right) \overset{(d)}{\rightarrow} Z, \quad \mbox{as } N\to\infty,
	\end{equation}
	where $Z$ is a centered normally distributed random variable with variance $\sigma ^{2}>0$. The notation  $\overset{(d)}{\rightarrow}$ stands for the convergence in distribution. 
	Motivated by (\ref{19a-2}), we can define the following estimator for the viscosity parameter $\theta$
	\begin{equation}\label{est}
		\widehat{\theta}_{N,x}(u)= \frac{6}{\pi V_{N, x}(u)}, \mbox{ for every } N\geq 1.
	\end{equation}
	It follows immediately from (\ref{19a-2}) that $ 	\widehat{\theta}_{N,x}(u)$ is a consistent estimator for $\theta$, i.e. $	\widehat{\theta}_{N,x}(u)\to \theta $ in probability, as $ N\to \infty$. We will also deduce, via (\ref{19a-3}), that this estimator is asymptotically Gaussian in the sense that 
	\begin{equation}\label{eq:CLT-theta}
		\sqrt{N} \left( 	\widehat{\theta}_{N,x}(u)-\theta\right) \overset{(d)}{\rightarrow} Z_{1} \sim N(0, \sigma _{1} ^{2}), \quad \mbox{as } N\to \infty,
	\end{equation}
	where the variance $\sigma _{1}^2>0$ can be expressed in terms of $\sigma$ from (\ref{19a-2}). 
	Further, our purpose is to analyze the asymptotic independence between the renormalized estimator $	\widehat{\theta}_{N,x}(u)$ and the data from which it is constructed, i.e., the observations $(u(t_{i}, x), i=0,1,..., N)$. 
	As mentioned in the introduction, we will consider a part of this data, which is contained in the  random vector 
	\begin{equation}\label{ynx}
		\mathbf{Y}_{N, x}(u)= \left( u(t_{i}, x), i\in J_{N}\right),
	\end{equation}
	where $J_{N} \subset \{1,2,..,, N\}$ is such that 
	\begin{equation}\label{mn}
		1\leq m(N):=card (J_{N})\leq N.
	\end{equation}
	The random vector $\mathbf{Y}_{N, x}(u)$ may contain all the data (when $ J_{N}= \{1,2,..., N\}$) or only a part of it. An example is $J_{N}=\{1,2,..., [N ^{\gamma}]\}$ with $0\leq \gamma <1$. 
	
	To understand how the data contained in $\mathbf{Y}_{N, x}(u)$ affects the limit distribution of $	\widehat{\theta}_{N,x}(u)$, we will measure the Wasserstein distance between the following probability distributions 
	\begin{equation*}
		\mathbb{P}_{ (\overline{\theta}_{N,x}(u), \mathbf{Y}_{N, x}(u)))} \mbox{ and } \mathbb{P}_{Z_{1}}\otimes \mathbb{P}_{\mathbf{Y}_{N, x}(u)},
	\end{equation*}
	with $Z_{1} $ from \eqref{eq:CLT-theta} and with the notation 
	\begin{equation*}
		\overline{\theta}_{N,x}(u)=	\sqrt{N} \left( \widehat{\theta}_{N,x}(u)-\theta\right), \quad N\geq 1.
	\end{equation*}
	Since $	\overline{\theta}_{N,x}(u)$ is essentially defined in terms of the quartic variation $ V_{N,x}(u)$, our approach focuses on analyzing first the distance between the following probability distributions
	\begin{equation*}
		\mathbb{P}_{ (U_{N,x}(u), \mathbf{Y}_{N, x}(u))} \mbox{ and } \mathbb{P}_{Z}\otimes \mathbb{P}_{\mathbf{Y}_{N,x}(u)}, \quad N\geq 1,
	\end{equation*}
	with $ U_{N, x}(u)$ defined by (\ref{19a-3}) and $Z$ being the limit distribution in (\ref{19a-3}). This will be done by using the multidimensional Stein-Malliavin calculus introduced in \cite{Pi} and \cite{T2}, based on the bound (\ref{intro-1}).
	When the Wasserstein distance between the two probability distributions from above converges to zero as $N\to \infty$, we will conclude that the sequences $ (U_{N, x}(u), N\geq 1)$ and $(\mathbf{Y}_{N, x}(u), N\geq 1)$ are asymptotically independent (in the Wasserstein sense). As already mentioned, to pass from the asymptotic independence of $ (U_{N, x}(u), N\geq 1)$ and $(\mathbf{Y}_{N, x}(u), N\geq 1)$ to the asymptotic independence of $ (\overline{\theta}_{N, x}(u), N\geq 1)$ and $(\mathbf{Y}_{N, x}(u), N\geq 1)$, we need to introduce and rely on some weaker concepts of asymptotic independence, see Section \ref{sec4}.

	\section{Asymptotic behavior of the quartic variation}
	In this section, we focus on the limit behavior of the quartic variation sequence defined by \eqref{vnx}. 
	We prove that it satisfies a CLT and we deduce its rate of convergence under the Wasserstein distance. We  start by treating the case of the linear stochastic heat equation (i.e., the drift coefficient is such that  $b\equiv 0$ in (\ref{1}), in which case the solution is a Gaussian random field) and then we consider the semilinear case. 
	
	\subsection{ The quartic variation of the solution to the linear stochastic heat equation}
	
	If $b\equiv0$, the mild solution $u$ given by \eqref{mild} shall be further denoted by $u_0$ and it takes the explicit form 
	\begin{equation}\label{mild2}
		u_{0}(t,x)= \int_{0} ^{t} \int_{\mathbb{R}} G(\theta (t-s), x-y)W(ds, dy),
	\end{equation}
	with $G$ given by (\ref{G}).  
	Recall that it solves the stochastic heat equation with additive space-time white noise, i.e.,
	\begin{equation*}
		\frac{\partial u_{0}}{\partial t}(t,x)=\frac{\theta}{2} \Delta u_{0} (t,x) + \dot{W}(t,x), \hskip0.4cm t> 0, x\in \mathbb{R},
	\end{equation*} 
	with vanishing initial condition $u_{0}(0, x)=0$ for every $x\in \mathbb{R}$. Then $ (u_{0}(t,x), t\geq 0, x\in \mathbb{R})$ is a centered Gaussian random field and its covariance with respect to the time variable can be expressed as (see e.g., \cite{PoTr}, \cite{T3})
	\begin{equation*}\label{eq:cov-linear}
		\mathbb{E} \left[u_{0} (t,x) u_{0} (s, x)\right]= \frac{1}{\sqrt{2\pi \theta}}\left( \sqrt{ t+s}-\sqrt{\vert t-s \vert }\right), \quad s,t\geq 0, x\in \mathbb{R}.
	\end{equation*}
	Moreover,  for any $p\geq 1$, 
	\begin{equation}\label{20a-3}
		\mathbb{E}\left[\vert u_{0} (t,x)- u_{0}(s,x) \vert ^{p}\right] \leq C\vert t-s\vert ^{ \frac{p}{4}}, \quad s,t\geq 0.
	\end{equation}
	This implies that   for every $x\in \mathbb{R}$, there is a modification (denoted also by $u_{0}$), such that the temporal sample path $t\to u_{0} (t,x)$ is H\"older continuous of order $\delta$ for every $\delta \in \left(0, \frac{1}{4}\right)$.  We can also write, for $i=0,1,...,N-1$, 
	\begin{equation}\label{21a-1}
		u_{0}(t_{i+1}, x)- u_{0} (t_{i}, x)= I_{1} (g_{i,x}),
	\end{equation}
	where $I_{p}$ denotes  the multiple integral of order $p\geq 1$ with respect to the white-noise $W$ (see Section \ref{s:malliavin}) and 
	\begin{equation}\label{gix}
		g_{i,x}(s,y)= G (\theta (t_{i+1}-s), x-y)1_{ (0, t_{i+1})}(s)- G (\theta (t_{i}-s), x-y)1_{ (0, t_{i})}(s).
	\end{equation} 
	Let us also recall some useful estimates concerning the mild solution (\ref{mild2}).  
	For $N\geq 1$ and $x\in \mathbb{R}$ we use the notation
	\begin{equation}
		u_{0} (t_{i+1},x)- u_{0} (t_{i}, x)= \Delta _{i} u_{0} (x), \quad  0\leq i\leq N-1.\label{23s-4}
	\end{equation}

	\begin{lemma} Consider the random field $(u_{0}(t,x), t\geq 0, x\in \mathbb{R})$ given by (\ref{mild2}). 
		Let $N\geq 1$ and $ i=0,1,...,N-1$. 
		Then there exists a constant $C$ such that
		\begin{equation}
			\label{a1}
			\left| \mathbb{E}\left[(\Delta _{i} u_{0}(x))^{2}\right]- \frac{\sqrt{2}}{\sqrt{\pi \theta N}}\right| \leq C \frac{1}{N ^{2}},
			\quad \mbox{and} \quad \left| \mathbb{E}\left[(\Delta _{i} u_{0}(x))^{4}\right]- \frac{6}{\pi \theta N }\right| \leq C \frac{1}{ N ^{\frac{5}{2}}},
		\end{equation}
		and for $i\not=j$,
		\begin{equation}
			\label{a5}
			\left| \mathbb{E} \left[\Delta _{i} u_{0}(x)\Delta _{j} u_{0}(x)\right]\right| \leq C\frac{1}{N^{2}}, \quad N\geq 1.
		\end{equation}
	\end{lemma}
	\begin{proof}
		The estimates (\ref{a1}) and (\ref{a5}) have been proven in \cite[Section 3.2]{PoTr}. 
	\end{proof}
	
	Let us define, for $N\geq 1$ and $x\in \mathbb{R}$,
	\begin{equation}\label{vnx2}
		V_{N, x}(u_{0}) =\sum_{i=0} ^{N-1} \left( u_{0} (t_{i+1}, x)- u_{0}(t_{i}, x) \right) ^{4},
	\end{equation}
	with $ t_{i}, i=0,1,...,N$ given by \eqref{ti}.

	We  state the main result concerning the behavior of the sequence defined by (\ref{vnx2}). The result is partially known,  see Propositions 8.4 and 8.5 in \cite{T3} (see also \cite{Sw} or \cite{MahTud2}), but some estimates in the proof  are new and needed in the sequel.
	\begin{prop}\label{pp11}
		Consider the sequence $ \left( 	V_{N, x}(u_{0}) , N\geq 1\right)$ given by (\ref{vnx2}). Then \begin{equation}\label{22n-22}
			\lim_{N\to \infty} V_{N, x}(u_{0})= \frac{ 6}{\pi \theta}\quad  \mbox{ almost surely and in } L^{2} (\mathbb{P}).
		\end{equation}
		Moreover, for $N$ large enough, 
		\begin{equation}\label{20n-5}
			\mathbb{E} \left[\left| V_{N, x} (u_{0})- \frac{ 6}{\pi \theta}\right| ^{2}\right] \leq C\frac{1}{N}. 
		\end{equation}
	\end{prop}
	\begin{proof}
		By (\ref{21a-1}), $\Delta _{i}u_{0}(x)=I_{1}(g_{i,x})$, so, by using the product formula for multiple stochastic integrals (\ref{prod}),
		\begin{equation}
			(\Delta _{i}u_{0}(x)) ^{4}=I_{4}\left( g_{i,x} ^{\otimes 4}\right) + 6\Vert g_{i,x} \Vert ^{2} I_{2}\left( g_{i,x} ^{\otimes 2}\right)+3\Vert g_{i,x}\Vert ^{4},\label{23s-3}
		\end{equation}
		where
		\begin{equation*}
			\Vert g_{i,x}\Vert ^{2}=\mathbb{E} \left( u_{0}(t_{i+1}, x)- u_{0} (t_{i}, x) \right) ^{2}.
		\end{equation*}
		By (\ref{23s-3}), we can write 
		\begin{equation}\label{20n-1}
			V_{N, x}(u_{0})-\frac{6}{\pi \theta} = T_{1, N}+ T_{2, N} + \mathbb{E} V_{N, x}(u_{0})-\frac{6}{\pi \theta}, 
		\end{equation}
		with
		\begin{equation}\label{t12}
			T_{1, N}=\sum_{i=0} ^{N-1} I_{4} \left( g_{i, x} ^{\otimes 4}\right), \hskip0.3cm T_{2, N} =6\sum_{i=0} ^{N-1} \Vert g_{i,x} \Vert ^{2} I_{2}\left( g_{i,x} ^{\otimes 2}\right).
		\end{equation}
		By (\ref{a1}), 
		\begin{eqnarray}
			&&	\left|  \mathbb{E} \left[V_{N, x}(u_{0})\right]-\frac{6}{\pi \theta}\right| = \left| \sum _{i=0} ^{N-1} \left( \mathbb{E}\left[\Delta _{i} u_{0} (x)\right]-\frac{6}{\pi \theta N} \right) \right|\nonumber\\
			&&\leq \sum _{i=0} ^{N-1} \left| \mathbb{E}\left[\Delta _{i} u_{0} (x)\right]-\frac{6}{\pi \theta N} \right|\leq C \frac{1}{N ^{\frac{3}{2}}}.\label{20n-2}
		\end{eqnarray}
		Then, by (\ref{a1}) and (\ref{a5}),
		\begin{eqnarray}
			\mathbb{E} T_{1,N} ^{2}&=& 4! \sum_{i, j=0} ^{N-1} \langle g_{i, x}, g_{j, x} \rangle ^{4}= 4! \left( \sum _{i=0} ^{N-1} \Vert g_{i, x}\Vert ^{8} + \sum _{i,j=0; i\not=j}^{N-1}\langle g_{i, x}, g_{j, x} \rangle ^{4}\right)\nonumber\\
			&=&4! \left[ \sum _{i=0} ^{N-1} \left( \mathbb{E}\left[ (\Delta _{i} u_{0} (x)) ^{2}\right]\right) ^{4} +  \sum _{i,j=0; i\not=j}^{N-1} \left( \mathbb{E} \left[(\Delta _{i} u_{0} (x))(\Delta _{j} u_{0} (x))\right]\right) ^{4} \right]\nonumber\\
			&= &\frac{96}{\pi ^{2} \theta ^{2}} \frac{1}{N}+ \mathcal{O} \left( N ^{-\frac{5}{2}}\right)\label{20n-3},
		\end{eqnarray}
		and
		\begin{eqnarray}
			\mathbb{E} T_{2, N} ^{2} &=&72 \sum_{i,j=0} ^{N-1} \Vert g_{i, x}\Vert ^{2} \Vert g_{j, x} \Vert ^{2} \langle g_{i, x}, g_{j, x} \rangle ^{2} \nonumber\\
			&=&72 \left( \sum _{i=0} ^{N-1} \Vert g_{i, x} \Vert ^{8} + \sum_{i,j=0; i\not=j}^{N-1}  \Vert g_{i, x}\Vert ^{2} \Vert g_{j, x} \Vert ^{2} \langle g_{i, x}, g_{j, x} \rangle ^{2}\right) \nonumber\\
			&=& 72  \left[ \sum _{i=0} ^{N-1} \left( \mathbb{E} \left[(\Delta _{i} u_{0} (x)) ^{2}\right]\right) ^{4} \right.\nonumber\\
			&&\left. +  \sum_{i,j=0; i\not=j}^{N-1} \mathbb{E} \left[(\Delta _{i} u_{0} (x)) ^{2}\right] \mathbb{E} \left[(\Delta _{j} u_{0} (x)) ^{2}\right] \left( \mathbb{E} \left[(\Delta _{i} u_{0} (x))(\Delta _{j} u_{0} (x))\right]\right) ^{2} \right]\nonumber\\
			&= &\frac{288}{\pi ^{2} \theta ^{2}}\frac{1}{N}+\mathcal{O} \left( N ^{-\frac{5}{2}}\right).\label{20n-4}
		\end{eqnarray}
		By (\ref{20n-1}), (\ref{20n-2}), (\ref{20n-3}) and (\ref{20n-4}),
		\begin{equation*}
			\mathbb{E} \left[\left| V_{N, x} (u_{0})- \frac{ 6}{\pi \theta}\right| ^{2}\right] = \mathbb{E}\left[T_{1, N}^{2}\right]+ \mathbb{E} \left[T_{2, N} ^{2}\right] + \left| \mathbb{E} \left[V_{N, x} (u_{0})\right]- \frac{ 6}{\pi \theta}\right|^{2} \leq C\frac{1}{N}.
		\end{equation*}
		This shows the inequality (\ref{20n-5}) and in particular, the $ L^{2}(\Omega)$ convergence of $ V_{N, x}(u_{0})$. By the hypercontractivity property \eqref{hyper}, for all $p\geq 1$,
		\begin{equation*}
			\mathbb{E} \left[\left| V_{N, x} (u_{0})- \frac{ 6}{\pi \theta}\right| ^{p}\right] \leq C N ^{-\frac{p}{2}}.
		\end{equation*}
		Concerning the almost sure convergence, we have for $\gamma \in \left(0, \frac{1}{2}\right)$ and $p\geq 1$,
		\begin{eqnarray*}
			&&	\sum _{N\geq 1} \mathbb{P} \left( \left| V_{N, x} (u_{0})- \frac{ 6}{\pi \theta}\right| >N ^{-\gamma}\right) \leq\sum _{N\geq 1} N ^{\gamma p} 	\mathbb{E} \left[\left| V_{N, x} (u_{0})- \frac{ 6}{\pi \theta}\right| ^{p}\right] \\
			&&\leq C \sum _{N\geq 1}N ^{p(\gamma -\frac{1}{2})},
		\end{eqnarray*}
		and the last series is finite for $p$ large enough. The Borel-Cantelli lemma implies the almost sure convergence of $ V_{N, x}(u_{0}) $ to $\frac{6}{\pi \theta}$.
	\end{proof}
	
	For $N\geq 1$ and $x\in \mathbb{R}$, set
	\begin{equation}\label{unx2}
		U_{N, x} (u_{0})= 	\sqrt{N}\left( V_{N, x} (u_{0}) - \frac{6}{\pi \theta} \right).
	\end{equation}
	We know that the above sequence satisfies a CLT (see e.g., Proposition 8.5 in \cite{T3}). In addition, we will get its  speed of convergence to the normal distribution under the Wasserstein distance.  Let us recall the definition of the Wasserstein distance. 	Set 
	$$\mathcal{A}= \{ h: \mathbb{R}^{n}\to \mathbb{R},  h \mbox{ is Lipschitz continuous with } \Vert h\Vert _{Lip}\leq 1 \},$$
	and let $F,G$ be two $n$-dimensional random vectors  such that $F, G\in L ^{1}(\Omega)$. 
	Then the Wasserstein distance between the probability distributions of $F$ and $G$ is defined by 
	\begin{equation}
		\label{dw}
		d_{W}(\mathbb{P}_{F}, \mathbb{P}_{G})=\sup_{h\in \mathcal{A}} \left| \mathbb{E}\left[h(F)\right]- \mathbb{E} \left[h(G)\right]\right|.
	\end{equation}
	We denoted by $\Vert h\Vert_{Lip} $  the Lipschitz constant of $h$ given by 
	\begin{equation*}
		\Vert h\Vert _{Lip}= \sup_{x,y\in \mathbb{R} ^{n}, x\not=y} \frac{ \vert h(x)-h(y)\vert}{\Vert x-y\Vert_{\mathbb{R} ^{n}}},
	\end{equation*} 
	with $\Vert \cdot \Vert_{\mathbb{R} ^{n}}$ the Euclidean norm in $\mathbb{R} ^{n}$.  For further purposes, let us recall that the Kolmogorov distance between the laws of $F$ and $G$ is given by 
	\begin{equation}\label{dkol}
		d_{K} (\mathbb{P}_{F}, \mathbb{P}_{G}):= \sup_{z\in \mathbb{R}^n} \left| \mathbb{P}(F\leq z)- \mathbb{P}( G\leq z)\right|.
	\end{equation}
	If $Z\sim N(0,1)$, then the Kolmogorov distance is dominated by the Wasserstein distance in the following sense (see e.g., \cite{NP-book}, Appendix C, see also \cite{Gaunt} for a deep study of the relation between these distances)
	\begin{equation}\label{20i-1}
		d_{K} (\mathbb{P}_{F}, \mathbb{P}_{Z}) \leq 2 \sqrt{ d_{W}(\mathbb{P}_{F}, \mathbb{P}_{Z})}.
	\end{equation}
	
	We use the classical Stein-Malliavin bound obtained in \cite{NP1}. 
	That is, if $F \in\mathbb{D} ^{1, 2} $ is a centered random variable and $Z\sim N(0, \sigma ^{2})$, then 
	\begin{equation}\label{sm1}
		d_{W} \left( \mathbb{P}_{F}, \mathbb{P}_{Z}\right)\leq C\left[ \sqrt{ \vert \mathbb{E}\left[F^{2}\right]-\sigma ^{2} \vert }+  \mathbb{E} \left[\left| \sigma ^{2}-\langle DF, D(-L) ^{-1} F\rangle \right|\right]\right].  
	\end{equation}
	We have the following result.
	\begin{theorem}\label{tt11}
		Consider the sequence $(U_{N, x} (u_{0}),N\geq 1)$ defined by (\ref{unx2}). Let $Z\sim N(0, \sigma _{\theta}^{2})$, with 
		\begin{equation}\label{s2}
			\sigma _{\theta} ^{2}:= \frac{384}{\pi ^{2} \theta ^{2}}.
		\end{equation}
		Then there exists a constant $C$ such that
		\begin{equation}\label{10m-1}
			d_{W} \left( \mathbb{P}_{U_{N, x}(u_{0})}, \mathbb{P}_{Z}\right) \leq C\frac{1}{\sqrt{N}}, \quad N\geq 1.
		\end{equation}
		In particular, 
		\begin{equation*}
			U_{N, x} (u_{0}) \overset{(d)}{\rightarrow} Z\sim N(0, \sigma _{\theta} ^{2}).
		\end{equation*}
	\end{theorem}
	\begin{proof}
		We will use the Stein-Malliavin bound (\ref{sm1}). 
		We first need to center the random variable $U_{N, x}(u_{0})$, so, for $N\geq 1$ and $x\in \mathbb{R}$, let 
		\begin{equation}\label{23n-1}
			\bar{U}_{N, x}(u_{0})= U_{N, x}(u_{0}) - \mathbb{E} \left[U_{N, x}(u_{0})\right].
		\end{equation}
		By the triangle's inequality, 
		\begin{eqnarray*}
			&&d_{W}\left( \mathbb{P}_{ U_{N,x}(u_{0})}, \mathbb{P}_{Z}\right) \leq 	d_{W}\left( \mathbb{P}_{\bar{ U}_{N,x}(u_{0})}, \mathbb{P}_{Z}\right) + d_{W} \left( \mathbb{P}_{ U_{N,x}(u_{0})}, \mathbb{P}_{\bar{ U}_{N,x}(u_{0})}\right)\\
			&&\leq  d_{W}\left( \mathbb{P}_{\bar{ U}_{N,x}(u_{0})}, \mathbb{P}_{Z}\right) + \mathbb{E}\left[\vert  U_{N, x}(u_{0})- \bar{U}_{N, x}(u_{0})\vert\right] \\
			&& = d_{W}\left( \mathbb{P}_{\bar{ U}_{N,x}(u_{0})}, \mathbb{P}_{Z}\right) + \vert \mathbb{E}\left[U_{N, x}(u_{0})\right]\vert.  
		\end{eqnarray*}
		We use (\ref{20n-2}) to get the bound 
		\begin{equation}\label{23n-2}
			\vert \mathbb{E} \left[U_{N, x}(u_{0})\right]\vert = \sqrt{N} \left\vert \mathbb{E} \left[V_{N, x}(u_{0})\right]-\frac{6}{\pi \theta}\right\vert \leq C \frac{1}{N},
		\end{equation}
		and this leads to 
		\begin{equation}\label{23n-3}
			d_{W}\left( P_{ U_{N,x}(u_{0})}, P_{Z}\right) \leq d_{W}\left( P_{\bar{ U}_{N,x}(u_{0})}, P_{Z}\right) + C \frac{1}{N}.
		\end{equation}
		We now estimate  $d_{W}\left( \mathbb{P}_{\bar{ U}_{N,x}(u_{0})}, \mathbb{P}_{Z}\right)$. 
		By \eqref{sm1} we can write 
		\begin{eqnarray}
			&&	d_{W}\left( \mathbb{P}_{\bar{ U}_{N,x}(u_{0})}, \mathbb{P}_{Z}\right) \leq C\left[ \sqrt{ \mathbb{E} \left[\bar{U}_{N, x}(u_{0})^{2}\right]-\sigma _{\theta} ^{2}}\right.\nonumber\\
			&&\left. + 	\mathbb{E}\left[\left| \langle  D(-L) ^{-1} \bar{U}_{N, x}(u_{0}),  D \bar{U}_{N, x}(u_{0})\rangle -\mathbb{E}\langle  D(-L) ^{-1} \bar{U}_{N, x}(u_{0}),  D \bar{U}_{N, x}(u_{0})\rangle \right|\right]\vphantom{\int}\right].\label{22n-10}
		\end{eqnarray}
		
		First we notice that
		\begin{equation*}
			\mathbb{E} \left[\bar{U}_{N, x}(u_{0}) ^{2}\right]= \mathbb{E}\left[U_{N,x}(u_{0}) ^{2}\right] - ( \mathbb{E}\left[ U_{N, x}(u_{0})\right])^{2} =N( \mathbb{E} \left[V_{N, x}(u_{0})^{2}\right]- (\mathbb{E}\left[V_{N, x}(u_{0})\right])^{2}).
		\end{equation*}
		By the decomposition (\ref{20n-1}), we have, with $ T_{1, N}, T_{2, N}$ given by (\ref{t12}),
		\begin{equation*}
			\mathbb{E} \left[V_{N, x} (u_{0}) ^{2}\right] -( \mathbb{E} \left[V_{N, x}(u_{0})\right]) ^{2} = \mathbb{E} \left[T_{1, N}^{2}\right] + \mathbb{E} \left[T_{2, N}^{2}\right] =\sigma _{\theta } ^{2} \frac{1}{N}+ \mathcal{O} \left( N ^{-\frac{5}{2}}\right).
		\end{equation*}
		Consequently,
		\begin{equation}\label{22n-11}
			\left| \mathbb{E} \left[\bar{U}_{N, x}(u_{0})^{2}\right]-\sigma _{\theta} ^{2}\right| = N \left| 	\mathbb{E} \left[V_{N, x} (u_{0}) ^{2}\right] -( \mathbb{E} \left[V_{N, x}(u_{0})\right]) ^{2} \right|= \mathcal{O} \left( N ^{-\frac{3}{2}}\right).
		\end{equation}
		Let us now evaluate the quantity
		\begin{equation*}
			\mathbb{E}\left[\left| \langle  D(-L) ^{-1} \bar{U}_{N, x}(u_{0}),  D \bar{U}_{N, x}(u_{0})\rangle -\mathbb{E}\langle  D(-L) ^{-1} \bar{U}_{N, x}(u_{0}),  D \bar{U}_{N, x}(u_{0})\rangle \right|^{2}\right].
		\end{equation*}
		By (\ref{23s-3}), for every $N\geq 1$ and $x\in \mathbb{R}$,  the random variable $\bar{U}_{N,x}(u_{0})$ admits the following chaos expansion 
		\begin{equation*}
			\bar{U}_{N,x}(u_{0})=\sqrt{N} \sum_{i=0} ^{N-1} \left( I_{4} \left( g_{i,x} ^{\otimes 4}\right) + 6 \Vert g_{i,x}\Vert ^{2} I_{2} \left( g_{i,x} ^{\otimes 2}\right) \right),
		\end{equation*}
		where recall that
		\begin{equation*}
			\Vert g_{i,x}\Vert ^{2}=\mathbb{E} \left( u_{0}(t_{i+1}, x)- u_{0} (t_{i}, x) \right) ^{2}.
		\end{equation*}
		This gives
		\begin{equation*}
			D\bar{U}_{N, x}(u_{0})= \sqrt{N}\sum_{i=0} ^{N-1} \left( 4 I_{3}\left( g_{i,x} ^{\otimes 3}\right) + 2\Vert g_{i,x}\Vert ^{2} I_{1}(g_{i,x})\right) g_{i,x},  
		\end{equation*}
		and
		\begin{equation}\label{21a-5}
			D(-L) ^{-1} \bar{U}_{N, x}(u_{0})=\sqrt{N}\sum_{i=0} ^{N-1} \left(  I_{3}\left( g_{i,x} ^{\otimes 3}\right) + \Vert g_{i,x}\Vert ^{2} I_{1}(g_{i,x})\right) g_{i,x}.  
		\end{equation}
		Therefore
		\begin{eqnarray*}
			&&	\langle  D(-L) ^{-1} \bar{U}_{N, x}(u_{0}),  D \bar{U}_{N, x}(u_{0})\rangle \\
			&=&N \sum _{i, j=0} ^{N-1} \left( 4 I_{3}\left( g_{i,x} ^{\otimes 3}\right) + 2\Vert g_{i,x}\Vert ^{2} I_{1}(g_{i,x})\right) \left(  I_{3}\left( g_{j,x} ^{\otimes 3}\right) + \Vert g_{j,x}\Vert ^{2} I_{1}(g_{j,x})\right)\langle g_{i, x} , g_{j, x} \rangle \\
			&=& 4 S_{1, N} + 3 S_{2, N} + 2 S_{3, N},
		\end{eqnarray*}
		where 
		\begin{eqnarray}
			&&	S_{1, N}= N \sum_{i,j=0} ^{N-1} I_{3}\left( g_{i,x} ^{\otimes 3}\right)I_{3}\left( g_{j,x} ^{\otimes 3}\right)\langle g_{i, x} , g_{j, x} \rangle ,\nonumber\\
			&&S_{2, N}= N \sum_{i,j=0} ^{N-1} I_{3}\left( g_{i,x} ^{\otimes 3}\right) I_{1}(g_{j,x})\Vert g_{j, x}\Vert ^{2} \langle g_{i, x} , g_{j, x} \rangle,\nonumber\\
			&& S_{3,N}=N \sum_{i,j=0} ^{N-1} I_{1}(g_{i,x})I_{1}(g_{j,x})\Vert g_{i, x} \Vert ^{2})\Vert g_{j, x} \Vert ^{2} \langle g_{i, x} , g_{j, x} \rangle, \quad N\geq 1.\label{si}
		\end{eqnarray}
		So,
		\begin{eqnarray}
			&&\langle  D(-L) ^{-1} \bar{U}_{N, x}(u_{0}),  D \bar{U}_{N, x}(u_{0})\rangle -\mathbb{E}\langle  D(-L) ^{-1} \bar{U}_{N, x}(u_{0}),  D \bar{U}_{N, x}(u_{0})\rangle \nonumber\\
			&=& 4 (S_{1, N}- \mathbb{E} S_{1, N})+ 3(S_{2, N}- \mathbb{E} S_{2, N}) + 2 (S_{3, N}- \mathbb{E} S_{3, N}).\label{22n-3}
		\end{eqnarray}
		Let us deal with $S_{1, N}- \mathbb{E} S_{1, N}$. We write 
		\begin{eqnarray*}
			S_{1, N} &=& N \sum _{i=0} ^{N-1} I_{3} (g_{i, x}^{\otimes 3}) ^{2}\Vert g_{i, x}\Vert ^{2} + N \sum_{i,j=0; i\not =j}^{N-1}  I_{3} (g_{i, x}^{\otimes 3}) I_{3} (g_{j, x}^{\otimes 3})\langle g_{i, x}, g_{j, x} \rangle\\
			&=:& S_{1, N} ^{(1)} + S_{1, N}^{(2)}.
		\end{eqnarray*}
		By employing  again the product formula (\ref{prod}),
		\begin{eqnarray*}
			&&S_{1, N} ^{(1)} - \mathbb{E} S_{1, N} ^{(1)} \\
			&=&
			N \sum _{i=0} ^{N} \left[ I _{6} \left( g_{i, x} ^{\otimes 6}\right)\Vert g_{i, x}\Vert ^{2} + 9 I _{4} \left( g_{i, x} ^{\otimes 4}\right)\Vert g_{i, x}\Vert ^{4} +18 I _{2} \left( g_{i, x} ^{\otimes 2}\right)\Vert g_{i, x}\Vert ^{6} \right]. 
		\end{eqnarray*}
		Thus, by using the estimate (\ref{a1}),
		\begin{eqnarray}
			\mathbb{E} \left[\vert S_{1, N} ^{(1)} - \mathbb{E} S_{1, N} ^{(1)}\vert ^{2}\right]
			&&\leq C N ^{2}  \sum_{i,j=0} ^{N-1}\left[  \langle g_{i,x}, g_{j, x} \rangle ^{6}  \Vert g_{i, x} \Vert ^{2}\Vert g_{j, x} \Vert ^{2}\right.\nonumber
			\\
			&&\left. \quad +  \langle g_{i,x}, g_{j, x} \rangle ^{4}\Vert g_{i, x} \Vert ^{4}\Vert g_{j, x} \Vert ^{4}+ \langle g_{i,x}, g_{j, x} \rangle ^{2}\Vert g_{i, x} \Vert ^{6}\Vert g_{j, x} \Vert ^{6}\right]\nonumber\\
			&&\leq  C N^{2} \sum_{i=0} ^{N-1} \Vert g_{i, x} \Vert ^{16} \leq CN^{2} \sum_{i=0} ^{N-1} \left( \frac{1}{\sqrt{N}}\right) ^{8} \leq C\frac{1}{N}.\label{22n-1}
		\end{eqnarray}
		By (\ref{a1}), (\ref{a5}) and the Cauchy-Schwarz's inequality,
		\begin{eqnarray}
			\mathbb{E}\left[\vert S_{1, N}^{(2)}\vert\right] &\leq & N  \sum_{i,j=0; i\not =j}^{N-1} \left( \mathbb{E}  I_{3} (g_{i, x}^{\otimes 3}) ^{2}\right) ^{\frac{1}{2}}\left( \mathbb{E}  I_{3} (g_{j, x}^{\otimes 3}) ^{2}\right) ^{\frac{1}{2}}\vert \langle g_{i, x}, g_{j, x} \rangle\vert \nonumber\\
			&\leq & CN \sum_{i,j=0; i\not =j}^{N-1}\Vert g_{i,x} \Vert ^{3}
			\Vert g_{j,x} \Vert ^{3}\vert \langle g_{i, x}, g_{j, x} \rangle\vert \nonumber \\
			&\leq & CN \sum_{i,j=0; i\not =j}^{N-1} \frac{1}{N ^{\frac{3}{4}}} \frac{1}{N ^{\frac{3}{4}}}\frac{1}{N ^{2}} \leq  C \frac{1}{\sqrt{N}}.\label{22n-2}
		\end{eqnarray}
		We get, by (\ref{22n-1}) and (\ref{22n-2}), 
		\begin{equation}\label{22n-4}
			\mathbb{E} \left[\vert S_{1, N}- \mathbb{E}S_{1, N} \vert \right] \leq  C \frac{1}{\sqrt{N}}.
		\end{equation}
		We proceed in a similar way with the summand $S_{2, N}$. We have
		\begin{eqnarray*}
			S_{2, N}&=& N\sum _{i=0} ^{N-1} I_{3} (g_{i, x}^{\otimes 3})I_{1}(g_{i, x}) \Vert g_{i, x}\Vert ^{4} \\
			&&+ N \sum_{i,j=0; i\not=j}^{N-1} I_{3} (g_{i, x}^{\otimes 3})I_{1}(g_{j, x})\Vert g_{j, x}\Vert ^{2} \langle g_{i, x}, g_{j,x}\rangle \\
			&=& S_{2, N}^{(1)} + S_{2, N} ^{(2)}.
		\end{eqnarray*}
		We estimate $S_{2, N} ^{(2)}$ by (\ref{a1}) and (\ref{a5}), 
		\begin{eqnarray*}
			\mathbb{E}\left[\vert S_{2, N} ^{(2)}\vert\right] &\leq & N \sum_{i,j=0; i\not=j}^{N-1} \left( \mathbb{E}I_{3} (g_{i, x}^{\otimes 3})^{2} \right) ^{\frac{1}{2}} \left( \mathbb{E} I_{1}(g_{i, x})^{2}\right) ^{\frac{1}{2}} \Vert g_{j, x}\Vert ^{2}\vert \langle g_{i,x},g_{j,x}\rangle \vert \\
			&&\leq CN \sum_{i,j=0; i\not=j}^{N-1} \Vert g_{i,x}\Vert ^{3} \Vert g_{j,x}\Vert ^{3} \vert \langle g_{i,x},g_{j,x}\rangle \vert\leq  C \frac{1}{\sqrt{N}}.
		\end{eqnarray*}
		Concerning $ S_{2, N} ^{(1)}$, we first decompose it in chaos by writting 
		\begin{equation*}
			S_{2, N} ^{(1)}- \mathbb{E} S_{2, N} ^{(1)}= N \sum_{i=0} ^{N-1} \left[ I_{4}(g_{i, x}^{\otimes 4}) +3 \Vert g_{i, x}\Vert ^{2} I_{2}(g_{i,x}^{\otimes 2}) \right] \Vert g_{i, x}\Vert ^{4}
		\end{equation*}
		and then, for $N$ large,  
		\begin{equation*}
			\mathbb{E} \left[\left| S_{2, N} ^{(1)}- \mathbb{E} S_{2, N} ^{(1)}\right| ^{2}\right] \leq C N ^{2} \sum_{i=0} ^{N-1}\Vert g_{i, x} \Vert ^{16}\leq C\frac{1}{N}.
		\end{equation*}
		We obtain 
		\begin{equation}\label{22n-5}
			\mathbb{E} \left[\vert S_{2, N}  - \mathbb{E} S_{2, N} \vert\right]\leq C\frac{1}{\sqrt{N}}.
		\end{equation}
		Finally
		\begin{eqnarray*}
			S_{3, N}&=& N\sum _{i=0} ^{N-1} I_{1}(g_{i, x}) ^{2} \Vert g_{i, x}\Vert ^{6} +N \sum_{i,j=0; i\not=j}^{N-1} I_{1}(g_{i,x})I_{1}(g_{j,x})\Vert g_{i,x} \Vert ^{2} \Vert g_{j,x}\Vert ^{2} \langle g_{i,x}, g_{j,x}\rangle \\
			&=& S_{3, N}^{(1)}+ S_{3, N} ^{(2)},
		\end{eqnarray*}
		and as above, for $N\geq 1$
		\begin{equation*}
			\mathbb{E} \left[\vert S_{3, N} ^{(2)}\vert \right] \leq CN  \sum_{i,j=0; i\not=j}^{N-1} \Vert g_{i, x}\Vert ^{3} \Vert g_{j, x}\Vert ^{3} \vert \langle g_{i,x}, g_{j,x}\rangle \vert \leq C\frac{1}{\sqrt{N}},
		\end{equation*}
		and 
		\begin{equation*}
			\mathbb{E}\left[\vert S_{3, N} ^{(1)}-\mathbb{E} S_{3, N} ^{(1)}\vert ^{2} \right]\leq cN ^{2} \sum_{i=0}^{N-1} \Vert g_{i, x}\Vert ^{16}\leq C\frac{1}{N},
		\end{equation*}
		so
		\begin{equation}\label{22n-6}
			\mathbb{E} \left[\vert S_{3, N}-\mathbb{E} S_{3, N}\vert\right] \leq C\frac{1}{\sqrt{N}}.
		\end{equation}
		By using (\ref{22n-3}) and the inequalities (\ref{22n-4}), (\ref{22n-5}) and (\ref{22n-6}), we conclude that
		\begin{eqnarray}
			&&\mathbb{E} \left[\left| 	\langle  D(-L) ^{-1} \bar{U}_{N, x}(u_{0}),  D \bar{U}_{N, x}(u_{0})\rangle -\mathbb{E}\langle  D(-L) ^{-1} \bar{U}_{N, x}(u_{0}),  D \bar{U}_{N, x}(u_{0})\rangle\right|\right]\nonumber \\
			&&\leq C \left( \mathbb{E}\left[\vert S_{1, N}-\mathbb{E}\left[S_{1, N}\right]\vert\right] + \mathbb{E}\left[\vert S_{2, N}-\mathbb{E}\left[S_{2, N}\right]\vert\right]+ \mathbb{E}\left[\vert S_{3, N}-\mathbb{E}\left[S_{3, N}\right]\vert\right]\right)\nonumber \\
			&&\leq C \frac{1}{\sqrt{N}}.\label{22n-12}
		\end{eqnarray}
		By plugging (\ref{22n-11}) and (\ref{22n-12}) into (\ref{22n-10}), we obtain 
		\begin{equation}
			\label{23n-4}d_{W}\left( \mathbb{P}_{\bar{ U}_{N,x}(u_{0})}, \mathbb{P}_{Z}\right) \leq C \frac{1}{\sqrt{N}},
		\end{equation} 
		and by combining this bound with (\ref{23n-3}), we get the conclusion.
	\end{proof}
	
	We notice that the inequality (\ref{10m-1}) can also be deduced by using the relation between the solution to the linear heat equation and the fractional Brownian motion, see e.g., \cite{MahTud2}. On the other hand, certain bounds obtained in the above proof (such as (\ref{22n-12})) are needed for the study of the asymptotic independence in the next section. 
	
	\subsection{The quartic variation of the solution to the semilinear stochastic heat equation}
	The next step is to transfer the results obtained  in the previous section  to the situation when the drift  coefficient $b$ does not vanish identically. 
	That is, we consider the SPDE (\ref{1}) by assuming that $b$ satisfies the assumption (\ref{cc}). 
	From (\ref{mild}), we can write the solution to (\ref{1}) as
	\begin{equation}\label{deco}
		u(t,x)= u_{0}(t,x)+ X(t,x), \quad t>0, x\in \mathbb{R},
	\end{equation}
	where $u_{0}$ is given by (\ref{mild2}) and 
	\begin{equation}\label{y}
		X(t,x)= \int_{0} ^{t} \int_{\mathbb{R}}G(\theta (t-s), x-y)b(u(s,y)) dy ds, \mbox{ for } t\geq 0, x\in \mathbb{R}. 
	\end{equation}
	The main observation is that the random field $X$, viewed as a stochastic process with respect to its time variable, is more regular than the Gaussian part $u_{0}$. 
	Actually, we have the following estimate (see Proposition 4.3.3 in \cite{DS1}): For every $T>0$, $p\geq 1$, and $\beta \in (0,1)$, we have 
	\begin{equation}\label{20a-1}
		\mathbb{E} \left[\vert X(t,x)- X(s, x) \vert ^{p}\right] \leq C_{T, p, \beta} \vert t-s \vert ^{\beta p}, \quad s, t \in [0, T],\; x\in \mathbb{R}.
	\end{equation}
	In particular, the paths $t\to X(t,x) $ are H\"older continuous of order $\delta $, for any $\delta \in (0, \beta)$. By using the estimate (\ref{20a-1}), we deduce the asymptotic behavior of the quartic variation for the solution to the semilinear heat equation and its rate of convergence under the Wasserstein distance. 
	\begin{prop}\label{pp4}
		Let $x\in \mathbb{R}$ be fixed and consider the sequence $(V_{N,x}(u), N\geq 1)$ given by (\ref{vnx}). Then, 
		\begin{equation*}
			\lim_{N\to \infty} V_{N, x}(u) =\frac{6}{\pi \theta} \quad \mbox{ in } L ^{1}(\mathbb{P}),
		\end{equation*}
		and
		\begin{equation*}
			U_{N, x}(u)=	\sqrt{N}\left( V_{N, x}(u)- \frac{6}{\pi \theta} \right) \overset{(d)}{\rightarrow} Z\sim N (0, \sigma _{\theta}^{2}) \quad \mbox{as } N\to \infty,
		\end{equation*}
		with $ \sigma_{\theta}$ given by \eqref{s2}. 
		Moreover, for every $\beta\in (3/4,1)$ there exists a constant $C>0$ such that
		\begin{equation}\label{20i-2}
			d_{W}( \mathbb{P}_{ U_{N, x}(u)}, \mathbb{P}_{Z})\leq C \left(\frac{1}{N}\right)^{\beta-3/4}, \quad \mbox{ for all }N\geq 1.
		\end{equation}
	\end{prop}
	\begin{proof}
		By using the decomposition (\ref{deco}) and setting
		\begin{equation}\label{pnx}
			P_{N,x} := \sum_{k=0}^{3} C_{4}^{k} \sum _{i=0} ^{N-1} \left( u_{0} (t_{i+1}, x) -u_{0}(t_{i}, x)\right) ^{k} \left( X(t_{i+1}, x) -X(t_{i}, x) \right) ^{4-k}, \quad N\geq 1,
		\end{equation}
		we can write 
		\begin{equation*}
			V_{N, x}(u)= V_{N, x}(u_{0}) + P_{N, x},
		\end{equation*}
		and
		\begin{equation*}
			\sqrt{N} \left( V_{N, x}(u)- \frac{6}{\pi \theta} \right) =	\sqrt{N} \left( V_{N, x}(u_{0})- \frac{6}{\pi \theta} \right) +\sqrt{N} P_{N,x}.
		\end{equation*}
		By using (\ref{20a-3}), (\ref{20a-1}) and the Cauchy-Schwarz's inequality, we get 
		\begin{eqnarray*}
			\mathbb{E} \left[\vert P_{N, x} \vert\right] &\leq &  \sum_{k=0}^{3} C_{4}^{k} \sum _{i=0} ^{N-1} \left( \mathbb{E} \left[\left( u_{0} (t_{i+1}, x) -u_{0}(t_{i}, x)\right) ^{2k}\right]\right) ^{\frac{1}{2}}
			\left(  \mathbb{E}  \left[\left( X(t_{i+1}, x) -X(t_{i}, x) \right) ^{2(4-k)}\right]\right) ^{\frac{1}{2}}\\
			&\leq & C \sum_{k=0}^{3} C_{4}^{k} \sum _{i=0} ^{N-1} \left( \frac{1}{N}\right) ^{\frac{k}{4}}  \left( \frac{1}{N}\right)  ^{(4-k) \beta}
			\leq C \sum_{k=0}^{3} C_{4}^{k} \left( \frac{1}{N}\right) ^{\frac{k}{4}+ (4-k) \beta -1}\\
			&=&C \sum_{k=0}^{3} C_{4}^{k} \left( \frac{1}{N}\right) ^{k\left(\frac{1}{4}-\beta\right)+ 4\beta -1}
			\leq C \sum_{k=0}^{3} C_{4}^{k} \left( \frac{1}{N}\right) ^{\beta-1/4}, \quad \mbox{for } \beta>1/4.
		\end{eqnarray*}
		Thus 
		\begin{equation}\label{24s-1}
			\sqrt{N} \mathbb{E} \left[\vert P_{N, x}\vert\right] \leq C \left( \frac{1}{N}\right) ^{\beta-\frac{3}{4}}, \quad \mbox{for } \beta>1/4.
		\end{equation}
		Furthermore, since
		\begin{equation*}
			U_{N,x}(u)= U_{N, x}(u_{0})+ \sqrt{N} P_{N,x},
		\end{equation*}
		we can write, by (\ref{24s-1}) and  Theorem \ref{tt11},
		\begin{eqnarray*}
			d_{W}( \mathbb{P}_{ U_{N, x}(u)}, \mathbb{P}_{Z})&\leq & 	d_{W}( \mathbb{P}_{ U_{N, x}(u_{0})}, \mathbb{P}_{Z})+ d_{W} ( \mathbb{P}_{ U_{N, x}(u)}, \mathbb{P}_{ U_{N, x}(u_{0})})\\
			&\leq & d_{W}( \mathbb{P}_{ U_{N, x}(u_{0})}, \mathbb{P}_{Z}) +\sqrt{N} \mathbb{E} \left[\vert P_{N, x}\vert\right]\\
			& \leq &C \left( \frac{1}{\sqrt{N}}+  \left( \frac{1}{N}\right) ^{\beta-\frac{3}{4}}\right)\leq C \left( \frac{1}{N}\right) ^{\beta-\frac{3}{4}}, \quad \mbox{for } \beta>3/4.
		\end{eqnarray*}
	\end{proof}
	
	\section{Asymptotic independence for the quartic variation}
	In a second step of our work, we study the asymptotic independence between the quatic variation sequence (\ref{19a-3})  and the data (\ref{ynx}), as described in Section \ref{sec21}. Again, we start our analysis with the case of the linear stochastic heat equation. 
	
	\subsection{The linear stochastic heat equation}
	In this part, we analyze the dependency between the random sequence $U_{N, x}(u_{0})$ defined by (\ref{unx2}) and the data that compose this estimator. More exactly, for $x\in \mathbb{R}$ fixed, let us define the observation vector $\mathbf{Y}_{N,x}(u_{0})$ given by 
	\begin{equation}
		\label{ynx2}
		\mathbf{Y}_{N,x} (u_{0})=\left( u_{0} (t_{i}, x), i\in J_{N}\right),
	\end{equation}
	where  $J_{N}$ is a subset of $\{1,2,...,N\}$ which satisfies 
	\begin{equation}\label{hipj}
		1\leq card(J_{N}):=m(N)\leq N.
	\end{equation}
	So, $ \mathbf{Y}_{ N, x}(u_{0})$ contains a part of the data $ (u_{0} (t_{i}, x), i=1,..., N)$ from which the estimator is defined. It may contain all the data when $ J_{N}= \{1,2,..., N\}$.  
	The analysis of the  dependency  between the quartic variation $U_{N, x}(u_{0})$  and the observations contained in $\mathbf{Y}_{N, x}(u_{0})$ will be done by comparing the following probability distributions:
	\begin{equation*}
		\mathbb{P}_{( U_{N, x}( u_{0}), \mathbf{Y}_{N, x}(u_{0}))}\mbox{ and } \mathbb{P}_{Z}\otimes \mathbb{P}_{ \mathbf{Y}_{N, x}(u_{0})}.
	\end{equation*}
	where $Z\sim N(0, \sigma _{\theta } ^{2})$  (which represents the limit distribution of $U_{N, x}( u_{0})$, see Proposition \ref{pp4}). 
	When these two distributions are close  in some sense, for $N$ large, i.e., the joint law of $( U_{N, x}( u_{0}), \mathbf{Y}_{N, x}(u_{0}))$ behaves as the product of its marginals, we say that $U_{N, x}(u_{0})$ and $ \mathbf{Y}_{N, x}(u_{0})$ are asymptotically independent. 
	This will depend, in particular, on the number of the observations contained in the vector $ \mathbf{Y}_{N, x}(u_{0})$. 
	
	Let us introduce a first notion of asymptotic independence. 
	Some (weaker) versions of asymptotic independence will be discussed in \Cref{sec4}.
	
	\begin{definition}
		\label{de1} 
		Let $ (X_{N}, N\geq 1)$ and $(\mathbf{Y}_{N}, N\geq 1)$ be two sequences of random variables with values in $\mathbb{R}$ and $\mathbb{R} ^{m(N)}$, respectively. Assume $ X_{N}\overset{(d)}{\rightarrow}U$, where $U$ is an arbitrary random variable.  
		We say that these two sequences are $W$-asymptotically independent if
		\begin{equation*}
			\lim_{N\to \infty}d_{W} \left( \mathbb{P}_{ (X_{N}, \mathbf{Y}_{N})}, \mathbb{P}_{U}\otimes \mathbb{P}_{ \mathbf{Y}_{N}}\right) =0. 
		\end{equation*}
	\end{definition}
	
	As a matter of fact, we provide quantitative estimates for the above Wasserstein distance, our basic tool being the multidimensional Stein-Malliavin bound proved in \cite{T2}, which generalizes the classical bound (\ref{sm1}). 
	For the reader convenience, we restate this result:
	\begin{theorem}\label{tt1}
		Let $X$ be a  centered random variable in $\mathbb{D} ^{1,2}$ and let $\mathbf{Y}=(Y_{1},...,Y_{d})$ be such that $ Y_{j}\in \mathbb{D} ^{1,2}$ for all $j=1,...,d$. Let  $ Z\sim N(0,\sigma ^{2})$. Then
		\begin{equation}\label{ineq1}
			d_{W} \left( \mathbb{P}_{ (X, \mathbf{Y})},  \mathbb{P}_{Z}\otimes \mathbb{P} _{\mathbf{Y}}\right) \leq C \left( \mathbb{E} \left[\left| \sigma ^{2} -\langle D(-L) ^{-1} X, DX\rangle \right|\right] + \sum_{j=1} ^{d} \mathbb{E} \left[\left| \langle D(-L) ^{-1} X, DY_{j}\rangle \right|\right] \right),
		\end{equation}
		with $C>0$.
	\end{theorem}
	Let us state and prove the main result of this section. 
	
	\begin{theorem}\label{tt2}
		Let $ x\in \mathbb{R}$ be fixed. 	
		Consider the random sequences $( U_{N, x}(u_{0}), N\geq 1)$ and $ (\mathbf{Y}_{N, x}(u_{0}), N\geq 1)$ given by (\ref{unx2}) and (\ref{ynx2}), respectively. Let $Z \sim N(0, \sigma _{\theta } ^{2})$, with $ \sigma _{\theta } $ given by (\ref{s2}). 
		Then, there exists a constant $C>0$ such that 
		\begin{equation*}
			d_{W}\left(\mathbb{P}_{ (U _{N, x}(u_{0}),\mathbf{Y}_{N,x}(u_{0}))}, \mathbb{P}_{ Z}\otimes \mathbb{P}_{ \mathbf{Y}_{ N, x}(u_{0})}\right)\leq C \left( \frac{ m(N)}{N} \right) ^{\frac{1}{2}}, \quad N\geq 1.
		\end{equation*}
		In particular, if $\lim\limits_{N\to \infty}\frac{m(N)}{N}=0$, then the sequences $(U_{N, x}(u_{0}), N\geq 1 )$ and $(\mathbf{Y}_{N, x}(u_{0}), N\geq 1)$ are $W$-asymptotically independent. 
	\end{theorem}
	\begin{proof}
		The first step is to replace  $U_{N, x}(u_{0})$ by its centered version $\bar{U}_{N, x}(u_{0})$ given by (\ref{23n-1}). 
		By the triangle's inequality,
		\begin{eqnarray}
			d_{W} \left( 	P_{ (U_{N, x}(u_{0}), \mathbf{Y}_{N,x})(u_{0})}, P _{Z} \otimes P_{ \mathbf{Y}_{N, x}(u_{0})}\right) &\leq & 	d_{W} \left( 	P_{ (\bar{U}_{N, x}(u_{0}), \mathbf{Y}_{N,x})}, P _{Z} \otimes P_{ \mathbf{Y}_{N, x}(u_{0})}\right)\nonumber \\
			&&+  d_{W} \left( 	P_{ (U_{N, x}(u_{0}), \mathbf{Y}_{N,x})},	P_{ (\bar{U}_{N, x}(u_{0}), \mathbf{Y}_{N,x})}\right).\label{23s-15}
		\end{eqnarray}
		Now, by the definition of the Wasserstein distance (\ref{dw}) and the inequality (\ref{23n-2})
		
		\begin{eqnarray*}
			d_{W} \left( 	P_{ (U_{N, x}(u_{0}), \mathbf{Y}_{N,x})},	P_{ (\bar{U}_{N, x}(u_{0}), \mathbf{Y}_{N,x})}\right) &\leq & \mathbb{E} \left[\vert U_{N, x} (u_{0})- \bar{U}_{N, x}(u_{0})\vert\right] \\
			&\leq& \vert \mathbb{E} \left[U_{N,x}(u_{0})\right]\vert\leq C\frac{1}{N}.
		\end{eqnarray*}
		So
		\begin{equation}\label{23n-15}
			d_{W} \left( 	P_{ (U_{N, x}(u_{0}), \mathbf{Y}_{N,x})(u_{0})}, P _{Z} \otimes P_{ \mathbf{Y}_{N, x}(u_{0})}\right) \leq  	d_{W} \left( 	P_{ (\bar{U}_{N, x}(u_{0}), \mathbf{Y}_{N,x})}, P _{Z} \otimes P_{ \mathbf{Y}_{N, x}(u_{0})}\right)+ C \frac{1}{N}.
		\end{equation}
		We now estimate $	d_{W} \left( 	P_{ (\bar{U}_{N, x}(u_{0}), \mathbf{Y}_{N,x})}, P _{Z} \otimes P_{ \mathbf{Y}_{N, x}(u_{0})}\right)$.  To do this, we apply the Stein-Malliavin bound in Theorem \ref{tt1} to get 
		\begin{eqnarray}
			d_{W} &\left( 	P_{ (\bar{U}_{N, x}(u_{0}), \mathbf{Y}_{N,x}(u_{0}))}, P _{Z} \otimes P_{ \mathbf{Y}_{N, x}(u_{0})}\right) 
			\leq  C\bigg[  \mathbb{E} \left[\left| \langle D(-L) ^{-1} \bar{U}_{N, x}(u_{0}), D\bar{U}_{N, x}(u_{0})\rangle -\sigma ^{2}\right|\right]  \nonumber\\
			&\left.+ \sqrt{\sum _{q\in J_{N}}  \mathbb{E} \left[\langle D(-L) ^{-1} \bar{U}_{N, x}(u_{0}), Du(t_{q}, x)\rangle ^{2}\right]} \right]. \label{23n-12}
		\end{eqnarray}
		By (\ref{22n-11})  and (\ref{22n-22})  in  the proof of Theorem \ref{tt11}, 
		\begin{equation}\label{23n-10}
			\mathbb{E} \left[\left| \langle D(-L) ^{-1} \bar{U}_{N, x}(u_{0}), D\bar{U}_{N, x}(u_{0})\rangle -\sigma ^{2}\right|\right] \leq C \frac{1}{\sqrt{N}},
		\end{equation}
		so it remains to bound the quantity
		\begin{equation*}
			\sqrt{\sum _{q\in J_{N}}  \mathbb{E} \left[\langle D(-L) ^{-1} \bar{U}_{N, x}(u_{0}), Du(t_{q}, x)\rangle ^{2}\right]}.
		\end{equation*}
		We  first calculate 
		$$ \langle D(-L) ^{-1} \bar{U}_{N, x}(u_{0}), Du(t_{q}, x)\rangle, \quad q=1,..., N. $$
		By (\ref{21a-5}) and the product formula (\ref{prod}),
		\begin{eqnarray*}
			\langle D(-L) ^{-1} \bar{U}_{N, x}(u_{0}), Du(t_{q}, x)\rangle
			&=&\sqrt{N} \sum_{i=0} ^{N-1} \left[ I_{3} (g_{i,x} ^{\otimes 3}) + \Vert g_{i,x}\Vert ^{2}I_{1}(g_{i, x})\right] \langle g_{i,x}, h_{t_{q}, x}\rangle \\
			&=& \sqrt{N} \sum_{i=0} ^{N-1} \left[ I_{3} (g_{i,x} ^{\otimes 3}) + \Vert g_{i,x}\Vert ^{2}I_{1}(g_{i, x})\right] P_{N} (i,q),
		\end{eqnarray*}
		where
		\begin{equation*}
			h_{t_{q}, x}(s,y)=G(\theta (t_{q}-s), x-y)1_{ (0, t_{q})} (s)
		\end{equation*}
		and
		\begin{equation}\label{piq}
			P_{N}(i,q)= \langle g_{i,x}, h_{t_{q}, x}\rangle.
		\end{equation}
		Therefore,
		\begin{eqnarray*}
			&&\mathbb{E} \left[\langle D(-L) ^{-1} \bar{U}_{N, x}(u_{0}), Du(t_{q}, x)\rangle ^{2}\right] \\
			&=&N \sum _{i, j=0} ^{N-1} \left[ 3! \langle g_{i, x}, g_{j, x}\rangle ^{3} +\Vert g_{i,x}\Vert ^{2}\Vert g_{j,x}\Vert ^{2}\langle g_{i, x} , g_{, x}\rangle \right] P_{N} (i,q) P_{N}(j,q)\\
			&=& N\sum_{i=0} ^{N-1}  (3!+1)\Vert g_{i,x} \Vert ^{6}P_{N} (i,q)^{2}\\
			&&+N \sum_{i,j=0; i\not=j}^{N}\left[ 3! \langle g_{i, x}, g_{j, x}\rangle ^{3} +\Vert g_{i,x}\Vert ^{2}\Vert g_{j,x}\Vert ^{2}\langle g_{i, x} , g_{j, x}\rangle \right] P_{N} (i,q) P_{N}(j,q).
		\end{eqnarray*}
		with $ P_{N}(i,q)$ given by (\ref{piq}). We write
		\begin{equation*}
			\mathbb{E} \left[\langle D(-L) ^{-1} \bar{U}_{N, x}(u_{0}), Du(t_{q}, x)\rangle ^{2}\right]= t_{1, N}(q)+ t_{2, N}(q),
		\end{equation*}
		where
		\begin{equation*}
			t_{1, N}(q)= 25N\sum _{i=0} ^{N-1}\Vert g_{i,x} \Vert ^{6}P_{N} (i,q)^{2},
		\end{equation*}
		and
		\begin{equation*}
			t_{2, N}(q)= N \sum_{i,j=0; i\not=j}^{N}\left[ 3! \langle g_{i, x}, g_{j, x}\rangle ^{3} +\Vert g_{i,x}\Vert ^{2}\Vert g_{j,x}\Vert ^{2}\langle g_{i, x} , g_{j, x}\rangle \right] P_{N} (i,q) P_{N}(j,q).
		\end{equation*}
		By using (\ref{a1}) and the estimates (\ref{eq:P_Niq}) in \Cref{ll2} from below,
		\begin{eqnarray*}
			\vert t_{1, N}(q)\vert \leq  C \frac{1}{\sqrt{N}} \sum _{i=0} ^{N-1} P_{N}(i, q) ^{2} \leq C \frac{1}{\sqrt{N}}\sup_{i=0,1,..., N-1}\vert P_{N} (i,q)\vert \sum _{i=0} ^{N-1} \vert P_{N}(i, q)\vert \leq  C \frac{1}{N}.
		\end{eqnarray*}
		Concerning the summand $t_{2, N}(q)$, we write, by using (\ref{a5}) and \Cref{ll2} from below,
		\begin{equation*}
			\vert t_{2, N}(q)\vert \leq C \frac{1}{ N ^{2}} \sum_{i,j=0; i\not=j}^{N} \vert P_{N}(i, q)\vert \vert P_{N} (j, q)\vert \leq C \frac{1}{ N ^{2}}.
		\end{equation*}
		Hence 
		\begin{equation*}
			\sqrt{ 	\mathbb{E} \left[\langle D(-L) ^{-1} \bar{U}_{N, x}(u_{0}), Du(t_{q}, x)\rangle ^{2}\right]}\leq \sqrt{ \sum _{q\in J_{N}} (\vert t_{1, N}(q)\vert + \vert t_{1, N}(q)\vert )}\leq C \sqrt{ \frac{m(N)}{N}}.
		\end{equation*}
		With the above estimate and (\ref{23n-10}), we obtain
		\begin{equation*}
			d_{W} \left(P_{ (\bar{U}_{N, x}(u_{0}), \mathbf{Y}_{N,x}(u_{0}))}, P_{Z} \otimes P_{ \mathbf{Y}_{N, x}(u_{0})}\right) \leq C \sqrt{ \frac{m(N)}{N}},
		\end{equation*}
		and, by combining this with (\ref{23n-12}), we can conclude.
	\end{proof}

	The following result has been used in the proof  of  Theorem \ref{tt2}.
	\begin{lemma}\label{ll2}
		for every $N\geq 1$ and $1\leq i,q \leq N$, 
		\begin{equation*}
			P_{N}(i,q)=\frac{1}{\sqrt{2\pi \theta N}}\left( \sqrt{i+1+q}-\sqrt{i+q}-\sqrt{\vert i+1-q\vert}+ \sqrt{\vert i-q\vert }\right).
		\end{equation*}
		In particular,
		\begin{equation}\label{eq:P_Niq}
			\vert P_{N}(i,q)\vert \leq C \frac{1}{\sqrt{N}} \quad \mbox{and}
			\quad \sum _{i=0} ^{N-1} \vert P_{N}(i,q)\vert \leq C, \quad 1\leq i,q \leq N.
		\end{equation}
	\end{lemma}
	\begin{proof}By using the semigroup propery of the heat kernel,
		\begin{eqnarray*}
			P_{N}(i,q)
			&=& \int_{0} ^{ t_{i+1} \wedge t_{q}}da \int_{\mathbb{R}} dz G (\theta( t_{i+1}-a), x-z) G (\theta (t_{q}-a), x-z) \\
			&&-  \int_{0} ^{ t_{i} \wedge t_{q}}da \int_{\mathbb{R}} dz G (\theta( t_{i}-a), x-z) G (\theta (t_{q}-a), x-z) \\
			&=& \frac{1}{\sqrt{2\pi \theta}} \int_{0} ^{ t_{i+1} \wedge t_{q}}da (t_{i+1}+t_{q}-2a)^{-\frac{1}{2}}- \frac{1}{\sqrt{2\pi \theta}} \int_{0} ^{ t_{i} \wedge t_{q}}da (t_{i}+t_{q}-2a)^{-\frac{1}{2}}\\
			&=&  \frac{1}{\sqrt{2\pi \theta}} \left( \sqrt{t_{i+1}+t_{q}} -  \sqrt{t_{i}+t_{q}}-\sqrt{\vert t_{i+1}-t_{q}\vert}+ \sqrt{\vert t_{i}-t_{q}\vert}\right).
		\end{eqnarray*}
		Thus
		\begin{equation}
			\label{piq2}
			P_{N}(i,q)= \frac{1}{\sqrt{2\pi \theta N}}\left( \sqrt{i+1+q}-\sqrt{i+q}-\sqrt{\vert i+1-q\vert}+ \sqrt{\vert i-q\vert }\right). 
		\end{equation}
		The bound (\ref{eq:P_Niq}) is trivial with $ C=\frac{2}{\sqrt{2\pi \theta}}$. 
		To deal with (\ref{eq:P_Niq}), 
		\begin{eqnarray*}
			&&	\sum_{i=0}^{N-1} \vert P_{N}(i,q)\vert \\
			&\leq & \frac{1}{\sqrt{2\pi \theta N}}\left[ \sum _{i=0} ^{N-1} \left( \sqrt{i+q+1}-\sqrt{i+q}\right) + \sum _{i=0}^{N-1} \left| \sqrt{ \vert i+1-q\vert }-\sqrt{\vert i-q\vert }\right|  \right]\\
			&=&  \frac{1}{\sqrt{2\pi \theta N}}\left[ \sqrt{N+q}-\sqrt{q} +  \sum _{i=0}^{N-1} \left| \sqrt{ \vert i+1-q\vert }-\sqrt{\vert i-q\vert }\right| \right].
		\end{eqnarray*}
		We compute the sum in the right-hand side above as follows:
		\begin{eqnarray*}
			&& 	 \sum _{i=0}^{N-1} \left|  \sqrt{ \vert i+1-q\vert }-\sqrt{\vert i-q\vert }\right| \\
			&&= \sum_{i=q+1} ^{N-1} \left( \sqrt{i+1-q} -\sqrt{i-q}\right)+ 1 +\sum _{i=0} ^{q-1} \left( \sqrt{q-i} -\sqrt{q-i-1}\right)\\
			&&=\sqrt{N-q} -\sqrt{q}. 
		\end{eqnarray*}
		Therefore,
		\begin{equation*}
			\sum_{i=0}^{N-1} \vert P_{N}(i,q)\vert \leq C \frac{1}{\sqrt{N} } \left( \sqrt{N+q}-2\sqrt{q}+\sqrt{N-q}\right) \leq C. 
		\end{equation*}
	\end{proof}

	\subsection{The semilinear stochastic heat equation}
	In what follows we analyze the asymptotic independence between the renormalized quartic variation $ U_{N, x}(u)$ of the solution (\ref{mild}) and the observation vector. 
	More precisely, we extend \Cref{tt2} for the linear case to the semilinear one, as follows:
	\begin{theorem}\label{tt3}
		Consider the random sequences $ (U_{N, x}(u), N\geq 1)$ and $ (\mathbf{Y}_{N,x}(u), N\geq 1)$ given by (\ref{19a-3}) and (\ref{ynx}), respectively. 
		Then, for every $\beta \in (3/4,1),$ there exists a constant $C>0$ such that
		\begin{equation}\label{24s-11}
			d_{W}\left( P_{ (U_{N,x}(u), \mathbf{Y}_{N, x}(u))}, P_{Z}\otimes P_{ \mathbf{Y}_{N,x}(u)}\right)\leq C\left(  \left( \frac{m(N)}{N} \right) ^{\frac{1}{2}}+ \frac{ m(N)}{N^{\beta}}+\left( \frac{1}{N}\right) ^{\beta-\frac{3}{4}}\right), \quad N\geq 1.
		\end{equation}
		In particular, if 
		\begin{equation}
			\label{24n-1}
			\lim\limits_{N\to \infty}\frac{ m(N)}{N^\gamma}=0 \quad \mbox{for some } \gamma\in (0,\beta),
		\end{equation}
		then $(U_{N, x}(u), N\geq 1)$ and $ (\mathbf{Y}_{N,x}(u), N\geq 1)$ are $W$-asymptotically independent.
	\end{theorem}
	\begin{proof}
		First, we use the triangle's inequality to write
		\begin{eqnarray*}
			&&	d_{W}\left( P_{ (U_{N,x}(u), \mathbf{Y}_{N, x}(u))}, P_{Z}\otimes P_{ \mathbf{Y}_{N,x}(u)}\right)
			\leq  d_{W}\left( P_{ (U_{N,x}(u), \mathbf{Y}_{N, x}(u))}, P_{ (U_{N, x}(u_{0}), \mathbf{Y}_{N,x}(u))}\right)\\
			&&+ d_{W}\left( P_{ (U_{N, x}(u_{0}), \mathbf{Y}_{N,x}(u))}, P_{ (U_{N, x}(u_{0}), \mathbf{Y}_{N,x}(u_{0}))}\right)
			+ d_{W} \left( P_{ (U_{N, x}(u_{0}), \mathbf{Y}_{N,x}(u_{0}))}, P_{Z}\otimes P_{\mathbf{Y}_{N,x}(u_{0})}\right)\\
			&&+d_{W} \left(  P_{Z}\otimes P_{\mathbf{Y}_{N,x}(u_{0})}, P_{Z}\otimes P_{\mathbf{Y}_{N,x}(u)}\right).
		\end{eqnarray*}
		By using the definition of the Wasserstein distance  (\ref{dw}), we deduce
		\begin{eqnarray*}
			&&	d_{W}\left( P_{ (U_{N,x}(u), \mathbf{Y}_{N, x}(u))}, P_{Z}\otimes P_{ \mathbf{Y}_{N,x}(u)}\right)\\
			&&\leq \mathbb{E} \left[\vert U_{N,x}(u)- U_{N,x}(u_{0})\vert\right] + \mathbb{E} \left[\Vert \mathbf{Y}_{N,x}(u)-\mathbf{Y}_{N,x}(u_{0})\Vert _{1}\right]\\
			&&+d_{W} \left( P_{ (U_{N, x}(u_{0}), \mathbf{Y}_{N,x}(u_{0}))}, P_{Z}\otimes P_{\mathbf{Y}_{N,x}(u_{0})}\right)+ \mathbb{E} \left[\Vert \mathbf{Y}_{N,x}(u)-\mathbf{Y}_{N,x}(u_{0})\Vert _{1}\right]\\
			&&\leq \sqrt{N} \mathbb{E}\left[\vert P_{N,x}\vert\right] +2\, \mathbb{E} \left[\Vert \mathbf{Y}_{N,x}(u)-\mathbf{Y}_{N,x}(u_{0})\Vert _{1}\right]
			+ d_{W} \left( P_{ (U_{N, x}(u_{0}), \mathbf{Y}_{N,x}(u_{0}))}, P_{Z}\otimes P_{\mathbf{Y}_{N,x}(u_{0})}\right),
		\end{eqnarray*}
		with $P_{N,x}$ given by (\ref{pnx}). By using (\ref{20a-1}), we notice that
		\begin{eqnarray}
			\mathbb{E} \left[\Vert \mathbf{Y}_{N,x}(u)-\mathbf{Y}_{N,x}(u_{0})\Vert _{1}\right]&=&\sum_{i\in J_{N}} \mathbb{E} \left[\vert X(t_{i+1}, x)- X(t_{i}, x)\vert\right] \nonumber \\
			&\leq & C \sum _{i\in J_{N}} \left( \frac{1}{N} \right) ^{\beta} \leq C \frac{m(N)}{N ^{\beta}},\label{24s-10}
		\end{eqnarray}
		for every $\beta \in (0,1)$.  The conclusion  (\ref{24s-11})  is obtained by (\ref{24s-10}), Theorem \ref{tt2} and (\ref{24s-1}).
	\end{proof}
	
	Let us end this section with some short comments.
	
	\begin{remark}
		\begin{itemize}
			\item[i)] In the semilinear case, the Wasserstein distance between the  renormalized quartic variation and its Gaussian limit  is the not same as in the linear case (compare Theorem \ref{tt11} with Proposition \ref{pp4}). That is, the perturbation denoted by $X$ and given by (\ref{y}) does  affect the speed of convergence of the quartic variation. 
			
			\item[ii)] Also, the condition for the asymptotic independence of the viscosity estimator from the observation vector is different in the semilinear case. For instance, if the cardinal of the data included in $\mathbf{Y}_{N,x}(u)$ is $m(M)=[N ^{\gamma}]$ then  (\ref{24n-1}) is satisfied if $0\leq  \gamma <\beta <1$  (instead of $ 0\leq \gamma<1$ in the linear case).   The perturbation (\ref{y}) may have an influence of the asymptotic independence  when $ \beta \leq \gamma <1$.  
		\end{itemize}
	\end{remark}

	\section{Drift parameter estimator: CLT and asymptotic independence with respect to the data}\label{sec4}
	Let us return to the viscosity parameter estimator $\widehat{\theta}_{N, x}$ defined in Section \ref{sec21} by \eqref{est}. 
	The goal now is to deduce its asymptotic properties, based on the results proved in the previous sections. 
	Although this estimator is strongly connected with the quartic variation $V_{N, x}(u)$, we will get weaker results for it. 
	This is due to the fact that, by the way it is constructed in (\ref{est}), its moments or its Malliavin derivative cannot be controlled as we did for the quartic variation. 
	
	\subsection{Central Limit Theorem for the viscosity  parameter estimator}
	We start by showing that the viscosity parameter estimator satisfies a qualitative Central Limit Theorem in the semilinear case. 
	\begin{theorem}\label{tt5}
		Let $ \widehat{\theta}_{N, x}(u)$ be given by (\ref{est}) with $ x\in \mathbb{R}$ fixed, and set
		\begin{equation}\label{21i-1}
			\overline{\theta}_{N, x}(u):=\sqrt{N} \left( 	\widehat{\theta}_{N,x}(u)-\theta\right), \hskip0.5cm N\geq 1. 
		\end{equation}
		Then the following assertions hold.
		\begin{enumerate}
			\item[(i)] $ \lim\limits_{N\to \infty}\widehat{\theta}_{N, x}(u)=\theta$ in probability.
			\item[(ii)] For every $t\in \mathbb{R}$ there exists a constant $c(t, \theta)$ depending on $t $ and $\theta$ such that
			\begin{equation*}
				\left| \mathbb{P} ( \overline{\theta}_{N, x} \leq t) - \mathbb{P}( Z_{1}\leq t) \right|  \leq c(t, \theta) N ^{-\frac{1}{4}}, \quad N\geq 1,
			\end{equation*}
			where $Z_{1}\sim N(0, \sigma _{1, \theta}^{2})$,  with $\sigma _{1, \theta}^{2}= \frac{\pi ^{2}\theta ^{4}}{36}\sigma _{\theta}^{2}$.
		\end{enumerate}
	\end{theorem}
	\begin{proof}
		Assertion (i) is a direct consequence of Proposition \ref{pp4}, so let us prove (ii).
		Let $t, x\in \mathbb{R}$ and $N\geq1$. 
		We write 
		\begin{equation*}
			\begin{aligned}
				\mathbb{P} ( \overline{\theta}_{N, x}(u) \leq t) 
				&= \mathbb{P} ( \widehat{\theta}_{N, x}(u) \leq \frac{t}{\sqrt{N}}+\theta)
				= \mathbb{P} \left( V_{N, x} (u) \geq \frac{6}{ \pi \left( \frac{t}{\sqrt{N}}+\theta \right)}\right)\\ 
				&= \mathbb{P} \left( U_{N, x}(u)\geq \frac{6}{\pi} \frac{-t}{\theta \left( \frac{t}{\sqrt{N}}+\theta \right)}\right).
			\end{aligned}
		\end{equation*}
		Thus
		\begin{equation*}
			\begin{aligned}
				&\left| \mathbb{P} ( \overline{\theta}_{N, x}(u) \leq t) - \mathbb{P}( Z_{1}\leq t) \right|
				= \left| \mathbb{P} \left( U_{N, x}(u)\geq \frac{6}{\pi} \frac{-t}{\theta \left( \frac{t}{\sqrt{N}}+\theta \right)}\right)-\mathbb{P}( Z_{1}\geq -t)\right| \\
				&= \left| \mathbb{P} \left( U_{N, x}(u)\leq \frac{6}{\pi} \frac{-t}{\theta \left( \frac{t}{\sqrt{N}}+\theta \right)}\right)-\mathbb{P}\left( \frac{\pi\theta^2}{6}Z \leq -t\right) \right|\\
				&\leq \left| \mathbb{P} \left( U_{N, x}(u)\leq \frac{6}{\pi} \frac{-t}{\theta \left( \frac{t}{\sqrt{N}}+\theta \right)}\right)-\mathbb{P} \left(Z\leq \frac{6}{\pi} \frac{-t}{\theta \left( \frac{t}{\sqrt{N}}+\theta \right)}\right)\right|\\
				& \quad + \left| \mathbb{P} \left( Z\leq \frac{6}{\pi} \frac{-t}{\theta \left( \frac{t}{\sqrt{N}}+\theta \right)}\right)-\mathbb{P}\left( \frac{\pi\theta^2}{6}Z \leq -t\right) \right|. 
			\end{aligned}
		\end{equation*}
		By majorizing the first summand from above by the Kolmogorov distance, we obtain, if $\tilde{Z} \sim N(0, 1)$ and $d_{K}$ is the Kolmogorov distance given by (\ref{dkol}),
		\begin{eqnarray*}
			&&\left| \mathbb{P} ( \overline{\theta}_{N, x} \leq t) - \mathbb{P}( Z_{1}\leq t) \right|\\
			&&\leq d_{K} \left( \mathbb{P}_{ U_{N, x}(u)}, \mathbb{P}_{ \tilde{Z}}\right) + \left| \mathbb{P} \left( \tilde{Z}\leq\frac{6}{\pi \sigma_{\theta}} \frac{-t}{\theta \left( \frac{t}{\sqrt{N}}+\theta \right)}\right)-\mathbb{P} \left( \tilde{Z} \leq \frac{ 6(-t) }{\pi \theta ^{2} \sigma _{\theta}}\right)\right|\\
			&&\leq  d_{K} \left( \mathbb{P}_{ U_{N, x}(u)}, \mathbb{P}_{ \tilde{Z}}\right) +\min \left( 1, \frac{1}{\sqrt{2\pi}} \left|\frac{6}{\pi \sigma_{\theta}} \frac{-t}{\theta \left( \frac{t}{\sqrt{N}}+\theta \right)}- \frac{ 6(-t) }{\pi \theta ^{2} \sigma _{\theta}}\right|\right),
		\end{eqnarray*}
		and, since $\sigma_{\theta} ^{2} =\frac{384}{\pi ^{2} \theta ^{2}}$ (see (\ref{s2})), we get 
		\begin{eqnarray*}
			\left| \mathbb{P} ( \overline{\theta}_{N, x} \leq t) - \mathbb{P}( Z_{1}\leq t) \right|
			&&\leq d_{K} \left( \mathbb{P}_{ U_{N, x}(u)}, \mathbb{P}_{ \tilde{Z}}\right) +\min \left( 1, \frac{1}{\sqrt{N}}\frac{6t^{2}}{\sqrt{2\pi  384}\theta \left| \frac{ t}{\sqrt{N}}+\theta \right| }\right)\\
			&&\leq 2d_{W} \left( \mathbb{P}_{ U_{N, x}(u)}, \mathbb{P}_{ \tilde{Z}}\right) +\min \left( 1, \frac{1}{\sqrt{N}}\frac{6t^{2}}{\sqrt{2\pi  384}\theta \left| \frac{ t}{\sqrt{N}}+\theta \right| }\right)\\
			&&\leq c(t,\theta) N ^{-\frac{1}{4}},
		\end{eqnarray*}
		where we used (\ref{20i-1}) and (\ref{20i-2}). 
	\end{proof}
	
	\subsection{Asymptotic independence for the viscosity parameter estimator}
	We now analyze the asymptotic independence between the estimator \eqref{est} and the observation vector \eqref{ynx}. 
	We cannot derive their asymptotic independence in the Wasserstein sense due to the difficulty to control the moments or the Malliavin derivative of the random variable $\widehat{\theta}_{N, x}(u)$. 
	Therefore, we need to introduce some weaker versions of asymptotic independence, as follows: 
	\begin{definition}\label{def2}
		Let $ (X_{N}, N\geq 1)$ be a real-valued sequence of random variables and let $Y_{N}$ be an $m(N)$-dimensional random vector, $N\geq 1$. 
		Assume $ X_{N}\overset{(d)}{\rightarrow}U$ for some random variable $U$. 
		Then $(X_N, N\geq 1)$ and $(Y_N, N\geq 1)$ are called
		\begin{enumerate}
			\item[i)] $K$-asymptotically independent if
			$
			\lim\limits_{N\to \infty} d_K(\mathbb{P}_{ X_{N}}, \mathbb{P}_{ Y_{N}})=0,
			$
			where $d_K$ is the Kolmogorov distance;
			more precisely, using the notation $\bar{t}_N:=\left(t_1,\dots t_{m(N)}\right)\in \mathbb{R}^{m(N)}$,
			\begin{equation*}
				d_K\left(\mathbb{P}_{ X_{N}}, \mathbb{P}_{ Y_{N}}\right):=	\sup_{t\in \mathbb{R}, \bar{t}_N\in \mathbb{R}^{m(N)}}
				\left| \mathbb{P} \left(X_{N}\leq t, Y_N\leq \bar{t}_N\right) - \mathbb{P}( U\leq t) \mathbb{P} \left(Y_N\leq \bar{t}_N\right)\right|, 
			\end{equation*}
			\item[ii)] $DK$-asymptotically independent if for every $t\in\mathbb{R}$,
			\begin{equation*}
				\lim_N\sup_{\bar{t}_N\in \mathbb{R}^{m(N)}}
				\left| \mathbb{P} \left(X_{N}\leq t, Y_N\leq \bar{t}_N\right) - \mathbb{P}( U\leq t) \mathbb{P}\left(Y_N\leq \bar{t}_N\right)\right|=0. 
			\end{equation*}
			\item[iii)] $ K_{1}$-asymptotically independent if 
			$
			\lim\limits_{N} d_{K_1}(\mathbb{P}_{ X_{N}}, \mathbb{P}_{ Y_{N}})=0,
			$
			where     
			\begin{equation*}
				\begin{aligned}
					d_{K_{1}}(\mathbb{P}_{ X_{N}}, \mathbb{P}_{Y_{N}})
					&:=	\sup_{t\in \mathbb{R}} \sup_{f\in \mathcal{A}_{N}}\left|\mathbb{E}\left[ 1_{(-\infty, t]}	(X_{N}) f(Y_{N})\right]- \mathbb{P}( U\leq t) \mathbb{E}\left[ f(Y_{N})\right]\right|\\
					\mathcal{A}_{N}&:=\left\{f: \mathbb{R} ^{m(N)}\to \mathbb{R} : \Vert f\Vert _{Lip}\leq 1, \Vert f\Vert _{\infty}\leq 1\right\}.
				\end{aligned}
			\end{equation*}
			\item[iv)] $DK_{1}$-asymptotically independent if for every $t\in \mathbb{R}$, 
			\begin{equation*} 
				\lim\limits_{N\to \infty}\sup_{f\in \mathcal{A}_{N} }
				\left| \mathbb{E}\left[ 1_{(-\infty, t]}	(X_{N}) f(Y_{N})\right]- \mathbb{P}(U\leq t) \mathbb{E}\left[ f(Y_{N})\right]\right|=0.
			\end{equation*}
		\end{enumerate}
	\end{definition}
	
	\begin{remark}
		Clearly, $K$(resp. $K_{1}$)-asymptotic independence implies the $DK$(resp. $DK_{1}$)-asymptotic independence. 
	\end{remark}
	
	Let us show that the  convergence in the Wasserstein distance implies the convergence under the above $DK_{1}$-distance. 
	\begin{prop}\label{pp33}
		Let $ (X_{N}, N\geq 1)$ and $(Y_{N}, N\geq 1)$ be two sequence of random variables with $X_N$ taking values in $\mathbb{R}$ and $Y_N$ taking values in $\mathbb{R} ^{m(N)}$, respectively. Let $Z$ be a random variable with density bounded by $M>0$. Then 
		\begin{eqnarray*}
			d_{K_{1}}(\mathbb{P}_{ (X_{N}, Y_{N})}, \mathbb{P}_{Z}\otimes \mathbb{P}_{ Y_{N}})&\leq& 6\sqrt{M} \sqrt{ d_{W} \left( \mathbb{P}_{ (X_{N}, Y_{N})}, \mathbb{P}_{ Z}\otimes \mathbb{P}_{ Y _{N}}\right) +2 d_{W} (\mathbb{P}_{ X_{N}}, \mathbb{P}_{Z})}\\
			&&+ d_{W} (\mathbb{P}_{ (X_{N}, Y_{N})}, \mathbb{P}_{Z}\otimes \mathbb{P}_{ Y_{N}}), \quad N\geq 1.
		\end{eqnarray*}
	\end{prop}
	\begin{proof}
		For $ t\in \mathbb{R}$ and $\eps>0$, we consider the function 
		\begin{equation*}
			\varphi_{t, \eps} (x)=\begin{cases}
				1, \mbox{ if } x\leq t\\\frac{t+\eps-x}{\eps}, \mbox{ if } t<x< t+\eps\\
				0, \mbox{ if } x\geq t+\eps,
			\end{cases}
		\end{equation*}
		so that $\Vert 	\varphi_{t, \eps}\Vert _{Lip}=\frac{1}{\eps}$, $ \Vert 	\varphi_{t, \eps}\Vert _{\infty}\leq 1$.  For $N \geq 1$  and $ f\in \mathcal{A}_{N}$, let
		\begin{equation*}
			E_{t}(N):= \left| \mathbb{E}\left[ 1_{(-\infty, t]}	(X_{N}) f(Y_{N})\right]- \mathbb{P}( X_{N}\leq t) \mathbb{E} \left[f(Y_{N})\right]\right|.
		\end{equation*}
		We bound this quantity as follows:
		\begin{eqnarray*}
			E_{t}(N) &\leq & \left| \mathbb{E} \left[\varphi_{t, \eps} (X_{N}) f( Y_{N})\right] - \mathbb{E} \left[\varphi_{t, \eps}(X_{N})\right] \mathbb{E} \left[f(Y_{N})\right] \right| \\
			&&+ \mathbb{E} \left[(\varphi_{t, \eps}-1_{(-\infty, t]}	)(X_{N} )\vert f\vert (Y_{N})\right] +  \mathbb{E} \left[(\varphi_{t, \eps}-1_{(-\infty, t]}	)(Z)\vert f\vert (Y_{N})\right].
		\end{eqnarray*}
		Since $f\in \mathcal{A}_{N}$, we notice that $\Vert \varphi_{t, \eps}f\Vert_{Lip}\leq \frac{1}{\eps}+1. $ Hence 
		\begin{eqnarray}
			E_{t}(N) \leq \frac{\eps+1}{\eps} d_{W} \left( \mathbb{P}_{ (X_{N}, Y_{N})}, \mathbb{P}_{Z}\otimes \mathbb{P}_{ Y_{N}}\right)+ 2\mathbb{P} ( X_{N} \in [t, t+\eps])+ 2\mathbb{P} (Z \in [t, t+\eps]).  \label{20i-4}
		\end{eqnarray}
		Let us now define the function $\Psi_{t, \eps}$ given by 
		\begin{equation*}
			\Psi_{t, \eps}(x) = \begin{cases}
				0, \mbox{ if } x\leq t-\eps \\
				\frac{x-t+\eps}{\eps}, \mbox{ if } x \in [t-\eps, t]\\
				1, \mbox{ if } x\in [t, t+\eps]\\
				\frac{ t+\eps -x}{\eps}, \mbox{ if } x \in [t+\eps, t+2\eps]\\
				0, \mbox{ if } x\geq t+2\eps.
			\end{cases}
		\end{equation*}
		This function satisfies $\Vert \Psi_{t, \eps} \Vert _{\infty} \leq 1$ and $\Vert \Psi_{t, \eps} \Vert_ {Lip}=\frac{1}{\eps}$. 
		Thus, (\ref{20i-4}) becomes 
		\begin{eqnarray*}
			E_{t}(N) 
			&&\leq  \frac{\eps+1}{\eps} d_{W} \left( \mathbb{P}_{ (X_{N}, Y_{N})}, \mathbb{P}_{Z}\otimes \mathbb{P}_{ Y_{N}}\right)+2\mathbb{E} \left[\Psi_{t, \eps}(X_{N})\right]+2\mathbb{P} (Z \in [t, t+\eps])\\
			&&\leq  \frac{\eps+1}{\eps} d_{W} \left( \mathbb{P}_{ (X_{N}, Y_{N})}, \mathbb{P}_{Z}\otimes \mathbb{P}_{ Y_{N}}\right)+2 \left| \mathbb{E} \left[\Psi_{t, \eps}(X_{N})\right] -\mathbb{E} \left[\Psi_{t, \eps}(Z)\right]\right| +  3\mathbb{E} \left[\Psi_{t, \eps}(Z)\right]\\
			&&\leq  \frac{\eps+1}{\eps} d_{W} \left( \mathbb{P}_{ (X_{N}, Y_{N})}, \mathbb{P}_{Z}\otimes P_{ Y_{N}}\right)+\frac{2}{\eps} d_{W} ( \mathbb{P}_{X_{N}}, \mathbb{P}_{Z}) + 3 \mathbb{P} (Z \in [t-\eps, t+2\eps])\\
			&&\leq\left( 1+\frac{1}{\eps}\right) d_{W} \left( \mathbb{P}_{ (X_{N}, Y_{N})}, \mathbb{P}_{Z}\otimes \mathbb{P}_{ Y_{N}}\right)+ \frac{2}{\eps} d_{W} ( \mathbb{P}_{X_{N}}, P_{Z}) + 9\eps M=: h(\eps). 
		\end{eqnarray*}
		So 
		\begin{equation*}
			E_{t}(N) \leq \inf _{\eps>0} h(\eps), \quad N\geq 1.
		\end{equation*}
		A straightforward computation shows that the above infimum is attained for 
		\begin{equation*}
			\eps:= \left( \frac{  d_{W} \left( \mathbb{P}_{ (X_{N}, Y_{N})}, \mathbb{P}_{Z}\otimes \mathbb{P}_{ Y_{N}}\right)+ 2 d_{W} ( \mathbb{P}_{X_{N}}, \mathbb{P}_{Z}) }{9M}\right) ^{\frac{1}{2}},
		\end{equation*}
		and this leads to the desired conclusion. \end{proof}
	
	As a consequence, we deduce a quantitative asymptotic independence of the quartic variation and the data under the $d_{K_{1}}$ distance (recall that the $W$-asymptotic independence has been quantified in \Cref{tt3}).
	\begin{corollary}\label{cor11}
		The following assertions hold.
		\begin{enumerate}
			\item[(i)] Let the notation of Theorem \ref{tt2} prevail. 
			Then there exists a constant $C>0$ such that
			\begin{equation*}
				d_{K_{1}}\left( \mathbb{P}_{ (U_{N,x}(u_{0}), \mathbf{Y}_{N, x}(u_{0}))}, \mathbb{P}_{Z}\otimes \mathbb{P}_{ \mathbf{Y}_{N,x}(u_{0})}\right)\leq C \left( \frac{m(N)}{N}\right) ^{\frac{1}{4}}, \quad N\geq 1.
			\end{equation*}
			\item[(ii)] Let the notation of Theorem \ref{tt3} prevail. 
			Then for every $\beta\in(3/4,1)$ there exists a constant $C>0$ such that
			\begin{equation*}
				\begin{aligned}
					&d_{K_{1}}\left(\mathbb{P}_{ (U_{N,x}(u), \mathbf{Y}_{N, x}(u))}, \mathbb{P}_{Z}\otimes \mathbb{P}_{ \mathbf{Y}_{N,x}(u)}\right)\\
					&\quad \leq C\left[\left( \frac{ m(N)}{N ^{\beta}}\right) ^{\frac{1}{2}} + \left( \frac{ m(N)}{N ^{\beta}}\right) ^{\frac{1}{4}} +\left( \frac{1}{N}\right) ^{\left(\beta-\frac{3}{4}\right)/4}\right], \quad N\geq 1.
				\end{aligned}
			\end{equation*}
		\end{enumerate}
	\end{corollary}
	\begin{proof}
		The proof follows immediately from \Cref{pp33}, \Cref{tt2} and \Cref{tt3}.    
	\end{proof}
	
	We are now in the position to derive the  $DK_1$-asymptotic independence for the viscosity  parameter with respect to the observed data, as introduced in \Cref{def2}, (iv).
	\begin{coro}\label{coro:K1-ai-theta}
		Let $x\in \mathbb{R}$,  $\overline{\theta}_{N, x}(u)$ be given by \eqref{21i-1}, and let $ (\mathbf{Y}_{N, x}(u), N\geq 1)$ be given by \eqref{ynx}.  
		Furthermore, let $\mathcal{A}_{N}$ be as in \Cref{def2} and $Z_{1}\sim N(0, \sigma _{1, \theta}^{2})$. 
		Then, for every $\beta\in(3/4,1)$ and every $t\in \mathbb{R}$, there exists $c(t, \theta)>0$ such that 
		\begin{equation}
			\begin{aligned}
				&\sup_{f\in \mathcal{A}_{N} }
				\left| \mathbb{E}\left[ 1_{(-\infty, t]}	(\overline{\theta}_{N,x}) f(\mathbf{Y}_{N,x}(u))\right]-\mathbb{P}(Z_{1}\leq t)\mathbb{E}\left[ f(\mathbf{Y}_{N,x}(u))\right]\right|\\
				&\quad \leq c(t,\theta)\left[\left( \frac{ m(N)}{N ^{\beta}}\right) ^{\frac{1}{2}} + \left( \frac{ m(N)}{N ^{\beta}}\right) ^{\frac{1}{4}} +\left( \frac{1}{N}\right) ^{\left(\beta-\frac{3}{4}\right)/4}\right], \quad N\geq 1.
			\end{aligned}
		\end{equation}
		In particular, if $\lim_{N\to \infty}\frac{m(N)}{N^\gamma}=0$ for some $\gamma\in (0,1)$, then the random sequences $(\overline{\theta}_{N, x}(u), N\geq 1)$ and $(\mathbf{Y}_{N, x}(u), N\geq 1)$ are $DK_{1}$-asymptotically independent.
	\end{coro}
	\begin{proof}
		For $t, x\in \mathbb{R}$ and $ N\geq 1$, we have
		\begin{equation*}
			\left( \overline{\theta}_{N, x}(u)\leq t\right) = \left( U_{N, x}(u)\geq \sqrt{N} \frac{ 6(-t) }{\pi \theta \left( \frac{t}{\sqrt{N}}+\theta \right) }\right).
		\end{equation*}
		So, for every $f\in \mathcal{A}_{N}$, 
		\begin{eqnarray*}
			&&\left| \mathbb{E}\left[ 1_{(-\infty, t]}	(\overline{\theta}_{N,x}) f(\mathbf{Y}_{N,x}(u))\right]-\mathbb{P}(Z_{1}\leq t)\mathbb{E}\left[ f(\mathbf{Y}_{N,x}(u))\right]\right| \\
			&=& \left| \mathbb{E}\left[ 1_{ \left[   -6t/\left[\pi \theta \left( \frac{t}{\sqrt{N}}+\theta \right) \right], \infty\right) }(U_{N, x}(u) )f(\mathbf{Y}_{N,x}(u))\right]-\mathbb{P}(Z_{1}\leq t)\mathbb{E}\left[ f(\mathbf{Y}_{N,x}(u))\right]\right|\\
			&\leq & \left| \mathbb{E}\left[ 1_{ \left[  -6t/\left[\pi \theta \left( \frac{t}{\sqrt{N}}+\theta \right) \right], \infty\right) }(U_{N, x}(u) )f(\mathbf{Y}_{N,x}(u))\right]-\mathbb{P}\left( Z \geq   \frac{ 6(-t) }{\pi \theta \left( \frac{t}{\sqrt{N}}+\theta \right) }\right) \mathbb{E}\left[ f(\mathbf{Y}_{N, x}(u))\right]\right| \\
			&&+ \left| \mathbb{P}(Z_{1} \leq t) - \mathbb{P}\left( Z \geq   \frac{ 6(-t) }{\pi \theta \left( \frac{t}{\sqrt{N}}+\theta \right) }\right)\right|,
		\end{eqnarray*}
		where $Z\sim N(0, \sigma _{\theta } ^{2})$. 
		Hence,
		\begin{eqnarray*}
			&&\left| \mathbb{E}\left[ 1_{(-\infty, t]}	(\overline{\theta}_{N,x}) f(\mathbf{Y}_{N,x}(u))\right]-\mathbb{P}(Z_{1}\leq t)\mathbb{E}\left[ f(\mathbf{Y}_{N,x}(u))\right]\right| \\
			&\leq & \left| \mathbb{E} \left[1_{ \left(-\infty,  \frac{ 6(-t) }{\pi \theta \left( \frac{t}{\sqrt{N}}+\theta \right) } \right] }(U_{N, x}(u) )f(\mathbf{Y}_{N,x}(u))\right]-\mathbb{P}\left( Z \leq   \frac{ 6(-t) }{\pi \theta \left( \frac{t}{\sqrt{N}}+\theta \right) }\right) \mathbb{E} \left[f(\mathbf{Y}_{N, x}(u))\right]\right| \\
			&&+ \left| \mathbb{P}(Z_{1} \leq t) - \mathbb{P}\left( Z \leq   \frac{ 6t }{\pi \theta \left( \frac{t}{\sqrt{N}}+\theta \right) }\right)\right|\\
			&\leq & d_{K_{1}} \left( \mathbb{P}_{ (U_{N, x}(u), \mathbf{Y}_{N, x}(u))}, \mathbb{P}_{Z} \otimes \mathbb{P}_{ \mathbf{Y}_{N, x}(u)}\right) +  \left| \mathbb{P}(Z_{1} \leq t) - \mathbb{P}\left( Z \leq   \frac{ 6t }{\pi \theta \left( \frac{t}{\sqrt{N}}+\theta \right) }\right)\right|.
		\end{eqnarray*}
		By \Cref{cor11}, the first summand in the right-hand side is bounded by  
		$$
		C\left[\left( \frac{ m(N)}{N ^{\beta}}\right) ^{\frac{1}{2}} + \left( \frac{ m(N)}{N ^{\beta}}\right) ^{\frac{1}{4}} +\left( \frac{1}{N}\right) ^{\left(\beta-\frac{3}{4}\right)/4}\right],
		$$
		whilst the second summand has been estimated in the proof of Theorem \ref{tt5}, it is actually  bounded by $c(t, \theta) \frac{1}{\sqrt{N}}$. 
		The claim in the statement now follows.
	\end{proof}
	
	\paragraph{$(D)K$-asymptotic independence for the viscosity parameter: The linear case}
	
	\vskip0.2cm
	
	\noindent Regarding the (notion of) asymptotic independence of the viscosity parameter $\overline{\theta}_{N,x}$ with respect to $\mathbf{Y}_{N,x}$, more can be said if we restrict ourselves to the linear case.
	More precisely, we can show that aside the $DK_1$-asymptotic independence obtained in \Cref{coro:K1-ai-theta}, one can also guarantee $DK$-asymptotic independence.
	This is fundamentally due to the fact that for the linear SPDE
	\begin{equation}\label{eq:SPDE-linear}
		\frac{\partial u_{0}}{\partial t}(t,x)=\frac{\theta}{2} \Delta u_{0} (t,x) + \dot{W}(t,x), \quad t>0, u(0,x)=0, x\in \mathbb{R},
	\end{equation} 
	the solution $(u_0(t,x))_{t\geq 0}$ given by \eqref{mild2} and hence
	$\mathbf{Y}_{N,x}$ are centered Gaussian, and for such random vectors the following inequality is known.
	\begin{prop}\label{prop:dk-w}
		There exists some universal constant $C>0$ such that for every $m\geq 1$, $F$ an arbitrary $\mathbb{R}^m$-valued random variable, and $N_\Sigma$ an $m$-dimensional centered Gaussian vector with invertible covariance matrix $\Sigma$, we have
		\begin{equation}
			d_K(F,N_\Sigma)\leq C\|\Sigma\|_{\rm HS}^{1/4}\;\sqrt{d_W(F,N_\Sigma)},
		\end{equation}
		where $\|\Sigma\|_{\rm HS}$ denotes the Hilbert-Schmidt norm of the matrix $\Sigma$.
	\end{prop}
	\begin{proof}
		It follows directly from \cite[{Proposition A.1 \& Remark A.2}]{NPYa22}
	\end{proof}
	Consequently, we derive the following:
	\begin{coro}\label{coro:ai-theta-linear}
		Let $x\in \mathbb{R}$, $\left(u_0(t,x)\right)_{t\geq 0}$ be the solution to the linear SPDE \eqref{eq:SPDE-linear}, given by \eqref{mild2},  $\overline{\theta}_{N, x}(u_0)$ be given by \eqref{21i-1}, and $ (\mathbf{Y}_{N, x}(u_0), N\geq 1)$ be given by \eqref{ynx}. 
		Furthermore,  $Z \sim N\left(0, \sigma _{\theta } ^{2}\right)$ with $ \sigma _{\theta } $ given by (\ref{s2}), and $Z_{1}\sim N\left(0, \sigma_{1, \theta}^{2}\right)$, whilst $\sigma_{1, \theta}$ is given in \Cref{tt5}.
		Then, there exists a constant $C_\theta$ that depends on $\theta$ such that  
		\begin{equation}\label{eq:U-ai-linear}
			d_{K}\left(P_{ (U _{N, x}(u_{0}),\mathbf{Y}_{N,x}(u_{0}))}, P _{ Z}\otimes P_{ \mathbf{Y}_{ N, x}(u_{0})}\right)\leq C_\theta \left( \frac{ m(N)}{\sqrt{N}} \right) ^{\frac{1}{2}}, \quad N\geq 1.
		\end{equation}
		Moreover, for every $t\in \mathbb{R}$ there exists a second constant $c(t,\theta)$ such that
		\begin{equation}\label{theta-ai-linear}
			\begin{aligned}
				&\sup_{\bar{t}_N\in \mathbb{R}^{m(N)}}
				\left| P \left(\overline{\theta}_{N,x}(u_{0})\leq t, \mathbf{Y}_{N, x}(u_{0})\leq \bar{t}_N\right) - P( Z\leq t) P \left(\mathbf{Y}_{N, x}(u_{0})\leq \bar{t}_N\right)\right|\\
				&\quad \leq C_\theta\left( \frac{ m(N)}{\sqrt{N}} \right) ^{\frac{1}{2}} + c(t,\theta)\frac{1}{\sqrt{N}}, \quad N\geq 1.
			\end{aligned}    
		\end{equation}
		In particular, if $\lim_{N\to \infty}\frac{m(N)}{\sqrt{N}}=0$ then the sequences $\left(U_{N,x}(u_0), N\geq 1\right)$ and $\left(\mathbf{Y}_{N,x}(u_0),N\geq 1\right)$ are $K$-asymptotically independent, whilst $\left(\overline{\theta}_{N,x}(u_0), N\geq 1\right)$ and $\left(\mathbf{Y}_{N,x}(u_0), N\geq 1\right)$ are $DK$-asymptotically independent.
	\end{coro}
	\begin{proof}
		First of all we note again that $\mathbf{Y}_{N,x}(u_0)$ is Gaussian and centered, 
		so by \Cref{tt2} and \Cref{prop:dk-w} we have
		\begin{equation*}
			d_{K}\left(P_{ (U _{N, x}(u_{0}),\mathbf{Y}_{N,x}(u_{0}))}, P _{ Z}\otimes P_{ \mathbf{Y}_{ N, x}(u_{0})}\right)\leq C \|\Sigma_N\|_{\rm HS}^{1/4}\left( \frac{ m(N)}{N} \right) ^{\frac{1}{4}},
		\end{equation*}
		where $\Sigma_N$ is the covariance matrix of the Gaussian vector $(Z,\mathbf{Y}_{N,x}(u_0))\in \mathbb{R}^{1+m(N)}$ with $Z$ independent of $\mathbf{Y}_{N,x}(u_0)$, $N\geq 1$.
		To handle $\|\Sigma_N\|_{\rm HS}$, recall that
		by \eqref{eq:cov-linear} we have
		\begin{equation*}
			\mathbb{E}\left[ u_{0} (t,x) u_{0} (s, x)\right]= \frac{1}{\sqrt{2\pi \theta}}\left( \sqrt{ t+s}-\sqrt{\vert t-s \vert }\right), \quad s,t\geq 0,
		\end{equation*}
		so taking into account that $\mathbf{Y}_{N,x}(u_0)$ is given by \eqref{ynx}, we have
		\begin{equation*}
			\begin{aligned}
				\|\Sigma_N\|^2_{\rm HS}
				&=\sigma_\theta^4+\sum\limits_{i,j\in J_N} \left(\mathbb{E}\left[ u_{0} (t_i,x) u_{0} (t_j, x) \right]\right)^2
				=\sigma_\theta^4+\frac{1}{2\pi \theta N}\sum\limits_{i,j\in J_N} \left( \sqrt{ i+j}-\sqrt{\vert i-j \vert }\right)^2\\
				&=\sigma_\theta^4+\frac{1}{2\pi \theta N}\left[\sum_{i\in J_N}2i+2\sum_{1<i\in J_N}\sum_{J_N\ni j<i}\left( \sqrt{ i+j}-\sqrt{i-j }\right)^2\right]\\
				&\leq\sigma_\theta^4+\frac{1}{2\pi \theta N}\left[2Nm(N)+2\sum_{1<i\in J_N}\sum_{J_N\ni j<i}2j\right]\\
				&\leq \sigma_\theta^4+\frac{2Nm(N)+4N\left[m(N)\right]^2}{\pi \theta N}\\
				&\leq \sigma_\theta^4+\frac{3\left[m(N)\right]^2}{\pi \theta}, \quad N\geq 1.
			\end{aligned}    
		\end{equation*}
		Now, inequality \eqref{eq:U-ai-linear} is immediate.
		
		To prove the second estimate \eqref{theta-ai-linear}, let us note that by following the same computational steps as in \Cref{coro:K1-ai-theta}, we get that for every $t\in \mathbb{R}$ there exists a constant $c(t,\theta)$ such that
		\begin{equation*}
			\begin{aligned}
				&\sup_{\bar{t}_N\in \mathbb{R}^{m(N)}}
				\left| P \left(\overline{\theta}_{N,x}(u_{0})\leq t, \mathbf{Y}_{N, x}(u_{0})\leq \bar{t}_N\right) - P( Z\leq t) P \left(\mathbf{Y}_{N, x}(u_{0})\leq \bar{t}_N\right)\right|\\
				&\quad\leq d_{K}\left(P_{ (U _{N, x}(u_{0}),\mathbf{Y}_{N,x}(u_{0}))}, P _{ Z}\otimes P_{ \mathbf{Y}_{ N, x}(u_{0})}\right)+\frac{c(t, \theta)}{\sqrt{N}}, \quad N\geq 1.
			\end{aligned}
		\end{equation*}
		Now, using the former estimate \eqref{eq:U-ai-linear} we get \eqref{theta-ai-linear}.
	\end{proof}
	
	\section{Appendix: Tools from Malliavin calculus}\label{s:malliavin}
	Here we briefly describe the elements from stochastic analysis that we need in the main body of the paper. 
	Consider $ \mathcal{H}$ a real separable Hilbert space and $(W(h), h \in  \mathcal{H})$ an isonormal Gaussian process on a probability space $(\Omega, {\cal{A}}, P)$, which is a centered Gaussian family of random variables such that ${\bf E}\left[ W(\varphi) W(\psi) \right]  = \langle\varphi, \psi\rangle_{ \mathcal{H}}$. Denote by  $I_{n}$ the multiple stochastic integral with respect to
	$W$ (see \cite{N}). This mapping $I_{n}$ is actually an isometry between the Hilbert space $ \mathcal{H}^{\odot n}$(symmetric tensor product) equipped with the scaled norm $\frac{1}{\sqrt{n!}}\Vert\cdot\Vert_{ \mathcal{H}^{\otimes n}}$ and the Wiener chaos of order $n$ which is defined as the closed linear span of the random variables $H_{n}(W(h))$ where $h \in  \mathcal{H}, \|h\|_{ \mathcal{H}}=1$ and $H_{n}$ is the Hermite polynomial of degree $n \in {\mathbb N}$
	\begin{equation*}
		H_{n}(x)=\frac{(-1)^{n}}{n!} \exp \left( \frac{x^{2}}{2} \right)
		\frac{d^{n}}{dx^{n}}\left( \exp \left( -\frac{x^{2}}{2}\right)
		\right), \hskip0.5cm x\in \mathbb{R}.
	\end{equation*}
	In particular, we have that $\mathbb{E}I_{n}(f)=0$ for all $f\in  \mathcal{H} ^{\otimes n}$ with $n\geq 1$ (see e.g. Lemma 1.1.1 in \cite{N}). The isometry of multiple integrals can be written as follows: for $m,n$ positive integers,
	\begin{eqnarray}
		\mathbb{E}\left(I_{n}(f) I_{m}(g) \right) &=& n! \langle \tilde{f},\tilde{g}\rangle _{ \mathcal{H}^{\otimes n}}\quad \mbox{if } m=n,\nonumber \\
		\mathbb{E}\left(I_{n}(f) I_{m}(g) \right) &= & 0\quad \mbox{if } m\not=n.\label{iso}
	\end{eqnarray}
	It also holds that
	\begin{equation*}
		I_{n}(f) = I_{n}\big( \tilde{f}\big),
	\end{equation*}
	where $\tilde{f} $ denotes the symmetrization of $f$.  We recall that any square integrable random variable which is measurable with respect to the $\sigma$-algebra generated by $W$ can be expanded into an orthogonal sum of multiple stochastic integrals
	\begin{equation}
		\label{sum1} F=\sum_{n=0}^\infty I_{n}(f_{n}),
	\end{equation}
	where $f_{n}\in  \mathcal{H}^{\odot n}$ are (uniquely determined)
	symmetric functions and $I_{0}(f_{0})=\mathbb{E}\left[  F\right]$.
	We will use the  hypercontractivity property of multiple stochastic integrals. Namely, if $F= \sum_{k=0} ^{n} I_{k}(f_{k}) $ with $f_{k}\in \mathcal{H} ^{\otimes k}$ then
	\begin{equation}
		\label{hyper}
		\mathbb{E}\vert F \vert ^{p} \leq C_{p} \left( \mathbb{E}F ^{2} \right) ^{\frac{p}{2}}.
	\end{equation}
	for every $p\geq 2$.


	For $p>1$ and $\alpha \in \mathbb{R}$ we introduce the Sobolev-Watanabe space $\mathbb{D}^{\alpha ,p }$  as the closure of
	the set of polynomial random variables with respect to the norm
	\begin{equation*}
		\Vert F\Vert _{\alpha , p} =\Vert (I -L) ^{\frac{\alpha }{2}} F \Vert_{L^{p} (\Omega )},
	\end{equation*}
	where $I$ represents the identity. We denote by $D$  the Malliavin  derivative operator that acts on smooth functions of the form $F=g(W(h_1), \dots , W(h_n))$ ($g$ is a smooth function with compact support and $h_i \in  \mathcal{H}$)
	\begin{equation*}
		DF=\sum_{i=1}^{n}\frac{\partial g}{\partial x_{i}}(W(h_1), \ldots , W(h_n)) h_{i}.
	\end{equation*}
	The operator $D$ is continuous from $\mathbb{D}^{\alpha , p} $ into $\mathbb{D} ^{\alpha -1, p} \left(  \mathcal{H}\right).$ The adjoint of $D$ is the divergence integral, denoted by $\delta$. It acts from $\mathbb{D} ^{\alpha -1, p} \left(  \mathcal{H}\right)$ onto $\mathbb{D}^{\alpha , p} $. The following duality relationship holds for every $F \in \mathbb{D}^{1,2}, u\in \mathbb{D}^{1,2}(\mathcal{H})$:
	\begin{equation}
		\label{dua}
		\mathbb{E} F \delta (u)= \mathbb{E} \langle DF, u\rangle_{\mathcal{H}}.
	\end{equation}
	
	The operator $L$, known as the
	generator of the Ornstein-Uhlenbeck semigroup, is defined as follows:
	\begin{equation*}
		LF=-\sum_{n\geq 0} nI_{n}(f_{n}),
	\end{equation*}
	if $F$ is given by (\ref{sum1}) such that $\sum_{n=1}^{\infty} n^{2}n! \Vert f_{n} \Vert ^{2} _{{\mathcal{H}}^{\otimes n}}<\infty$. Furthermore, we also define the pseudo-inverse of the operator $L$, denoted by $L^{-1}$, as follows: for each $F\in L^2(\Omega)$, we set $L^{-1} F=\sum_{n \geqslant 1}-\frac{1}{n} J_n(Z)$, where $J_n(Z)$ denotes the projection of $F$ onto the $n$th Wiener chaos. We mention that $L^{-1}$  is an operator with values in $\mathbb{D}^{2,2}$.  Moreover, for every $F\in L^2(\Omega)$, we have $L L^{-1} F=F-E(F)$, hence $L^{-1}$ acts as the inverse of the operator $L$ for centered random variables.

	We will intensively use the product formula for multiple integrals.
	It is well-known that for $f\in  \mathcal{H}^{\odot n}$ and $g\in  \mathcal{H}^{\odot m}$
	\begin{equation}\label{prod}
		I_n(f)I_m(g)= \sum _{r=0}^{n\wedge m} r! \left( \begin{array}{c} n\\r\end{array}\right) \left( \begin{array}{c} m\\r\end{array}\right) I_{m+n-2r}(f\otimes _r g),
	\end{equation}
	where $f\otimes _r g$ means the $r$-contraction of $f$ and $g$ (see e.g. Section 1.1.2 in \cite{N}). This contraction is defined, when $  \mathcal{H}= L^{2}(T, \mathbb{B}, \nu)$ (where $\nu$ is a sigma-finite measure without atoms)
	\begin{eqnarray}
		&&	\label{contra}
		(f \otimes _{r} g) (t_{1},...,t_{n+m-2r})\\
		&=& \int_{T^{r}} f(u_{1},...,u_{r}, t_{1},...,t_{n-r})g(u_{1},...,u_{r}, t_{n-r+1},...,t_{n+m-2r})d\nu(u_{1})....d\nu(u_{r}),\nonumber
	\end{eqnarray}
	for $r=1,..., n\wedge m$ and $f\otimes _{0}g=f\otimes g$, the tensor product. It holds that $f\otimes _{r}g\in  \mathcal{H} ^{\otimes n+m-2r}= L^{2}(T ^{n+m-2r})$. In general, the contraction $f\otimes _{r}g$ is not symmetric and we denote by $f\widetilde{\otimes } _{r} g$ its symmetrization.

\end{document}